\documentclass[11pt]{article}
\usepackage[T1]{fontenc}
\usepackage{lmodern}
\usepackage{amsmath,amsthm,amssymb,amsfonts}
\usepackage{enumerate}
\usepackage{epsfig}
\usepackage[dvipsnames,svgnames,table]{xcolor}
\usepackage{pgf,tikz}
\usetikzlibrary{calc}
\usepackage{fullpage}
\usepackage[colorlinks=true,linkcolor=RoyalBlue,urlcolor=RoyalBlue,citecolor=PineGreen]{hyperref}
\usepackage{booktabs}
\usepackage{cleveref}
\usepackage{enumitem}

\newtheorem{theorem}{Theorem}[section]
\newtheorem{proposition}[theorem]{Proposition}

\newtheorem{lemma}[theorem]{Lemma}
\newtheorem{corollary}[theorem]{Corollary}
\newtheorem*{remark}{Remark}

\newcommand{\geng}{\textsc{Geng}}
\newcommand{\nauty}{\textsc{Nauty}}
\newcommand{\mathematica}{\textsc{Mathematica}}

\colorlet{colfg}{black}
\colorlet{colgraphv}{colfg!75!white}
\colorlet{colgraphe}{colfg!55!white}
\tikzstyle{vertex}=[fill=colgraphv,circle,inner sep=0pt, minimum size=4pt]
\tikzstyle{edge}=[line width=1.5pt,colgraphe]

\title{Graph rigidity properties of Ramanujan graphs}
\author{Sebastian M. Cioab\u{a}\thanks{Department of Mathematical Sciences,
University of Delaware, Newark, DE 19716, USA. \newline E-mail: \texttt{cioaba@udel.edu}} 
\and Sean Dewar\thanks{Johann Radon Institute for Computational and Applied Mathematics (RICAM), Austrian Academy of Sciences, 4040 Linz, Austria. E-mail: \texttt{sean.dewar@ricam.oeaw.ac.at}}
\and Georg Grasegger\thanks{Johann Radon Institute for Computational and Applied Mathematics (RICAM), Austrian Academy of Sciences, 4040 Linz, Austria. E-mail: \texttt{georg.grasegger@ricam.oeaw.ac.at}} 
\and Xiaofeng Gu\thanks{Department of Computing and Mathematics, University of West Georgia, Carrollton, GA 30118, USA. \newline E-mail: \texttt{xgu@westga.edu}}
}

\begin{document}
\date{}
\maketitle

\begin{abstract}
A recent result of Cioab\u{a}, Dewar and Gu implies that any $k$-regular Ramanujan graph with $k\geq 8$ is globally rigid in $\mathbb{R}^2$. In this paper, we extend these results and prove that any $k$-regular Ramanujan graph of sufficiently large order is globally rigid in $\mathbb{R}^2$ when $k\in \{6, 7\}$, and when $k\in \{4,5\}$ if it is also vertex-transitive. These results imply that the Ramanujan graphs constructed by Morgenstern in 1994 are globally rigid. We also prove several results on other types of framework rigidity, including body-bar rigidity, body-hinge rigidity, and rigidity on surfaces of revolution. In addition, we use computational methods to determine which Ramanujan graphs of small order are globally rigid in $\mathbb{R}^2$.
\end{abstract}


\section{Introduction and main results}

In this paper, by a graph, we always mean a simple graph unless otherwise stated, and we also reserve the term multigraph for a graph with possible parallel edges but no loops.
A \textbf{$d$-dimensional framework} is a pair $(G, p)$, where $G$ is a graph and $p$ is a map from $V(G)$ to $\mathbb{R}^d$. Roughly speaking, it is a straight line realization of $G$ in $\mathbb{R}^d$. 
Given $\|\cdot\|$ is the Euclidean norm for $\mathbb{R}^d$,
we say two frameworks $(G, p)$ and $(G, q)$ are \textbf{equivalent} if $\|p(u) - p(v) \| = \|q(u) - q(v) \|$ holds for every edge $uv\in E(G)$,
and \textbf{congruent} if $\|p(u) - p(v) \| = \|q(u) - q(v) \|$ holds for every $u, v\in V(G)$. A framework $(G,p)$ is \textbf{generic} if the coordinates of its points are algebraically independent over the rationals. The framework $(G, p)$ is \textbf{rigid} if there exists $\varepsilon >0$ such that if $(G, p)$ is equivalent to $(G, q)$ and $\|p(u) - q(u) \| < \varepsilon$ for every $u\in V(G)$, then $(G, p)$ is congruent to $(G, q)$. As observed in \cite{AsRo78}, a generic realization of $G$ is rigid in $\mathbb{R}^d$ if and only if every generic realization of $G$ is rigid in $\mathbb{R}^d$. Hence, generic rigidity can be considered as a property of the underlying graph. Because of this,
a graph is defined to be \textbf{rigid} in $\mathbb{R}^d$ if every/some generic realization of $G$ is rigid in $\mathbb{R}^d$.

A $d$-dimensional framework $(G, p)$ is \textbf{globally rigid} if every framework that is equivalent to $(G, p)$ is congruent to $(G, p)$. It was proved in \cite{GortlerThurstonHealey10} that if there exists a generic framework $(G,p)$ in $\mathbb{R}^d$ that is globally rigid, then any other generic framework $(G,q)$ in $\mathbb{R}^d$ is also be globally rigid. Following from this, a graph $G$ is defined to be \textbf{globally rigid} in $\mathbb{R}^d$ if there exists a globally rigid generic framework $(G,p)$ in $\mathbb{R}^d$. A closely related concept to global rigidity is redundant rigidity. A graph $G$ is \textbf{redundantly rigid} in $\mathbb{R}^d$ if $G-e$ is rigid in $\mathbb{R}^d$ for every edge $e \in E(G)$. It was proved by Hendrickson \cite{Hendrickson92} that any globally rigid graph in $\mathbb{R}^d$ with at least $d+2$ vertices is $(d+1)$-connected and redundantly rigid in $\mathbb{R}^d$. Hendrickson \cite{Hendrickson92} also conjectured the converse. It can be shown easily that it is true for $d=1$, however the conjecture is not true for $d \geq 3$ \cite{Conn91}. The final case of the conjecture, i.e., when $d=2$, was confirmed to be true by the combination of a result of Connelly~\cite{Conn05} and a result of Jackson and Jord\'{a}n~\cite{JaJo05}. Thus, a graph $G$ is globally rigid in $\mathbb{R}^2$ if and only if $G$ is 3-connected and redundantly rigid, or $G$ is a complete graph on at most three vertices \cite{Conn05, JaJo05}.

Rigidity in $\mathbb{R}^2$ has been well studied. For a subset $X\subseteq V(G)$, let $G[X]$ be the subgraph of $G$ induced by $X$ and $E(X)$ denote the edge set of $G[X]$. A graph $G$ is \textbf{sparse} if $|E(X)|\le 2|X|-3$ for every $X\subseteq V(G)$ with $|X|\ge 2$. By definition, any sparse graph is simple. If in addition $|E(G)|=2|V(G)|-3$, then $G$ is called \textbf{$(2, 3)$-tight}. A graph $G$ is rigid in $\mathbb{R}^2$ if and only if $G$ contains a spanning $(2, 3)$-tight subgraph. This characterization was first discovered by Pollaczek-Geiringer \cite{Poll1927} and rediscovered by Laman \cite{Lama70}, and thus is also called the \textbf{Geiringer-Laman condition} in~\cite{Lovasz19}. A $(2, 3)$-tight graph is also called a \textbf{Laman graph}.

Lov\'{a}sz and Yemini \cite{LoYe82} gave a new characterization of rigid graphs and showed that $6$-connected graphs are rigid in $\mathbb{R}^2$. They also constructed infinitely many $5$-connected graphs that are not rigid in $\mathbb{R}^2$, showing that the connectivity condition was indeed tight. In fact, they proved a stronger result that every $6$-connected graph is rigid in $\mathbb{R}^2$ even with the removal of any three edges. This result, together with the combinatorial characterization of global rigidity mentioned above, implies that $6$-connected graphs are also globally rigid in $\mathbb{R}^2$ \cite[Theorem 7.2]{JaJo05}. This result was improved by Jackson and Jord\'an~\cite{JaJo09} using an idea of mixed connectivity, in which they showed that a simple graph $G$ is globally rigid in $\mathbb{R}^2$ if $G$ is $6$-edge-connected, $G-u$ is $4$-edge-connected for every vertex $u$ and $G-\{v, w\}$ is $2$-edge-connected for any vertices $v,w\in V(G)$.

By using a partition result of \cite{Gu18}, Cioab\u{a}, Dewar and Gu studied spectral conditions for rigidity and global rigidity in $\mathbb{R}^2$ in~\cite{CDG21}. The matrix $L(G)=D(G)-A(G)$ is called the \textbf{Laplacian matrix} of $G$, where $A(G)$ and $D(G)$ are the adjacency matrix and diagonal degree matrix of $G$, respectively. Let $G$ be a graph with $n$ vertices. For $1\leq i\leq n$, we use $\lambda_i(G)$ to denote the $i$-th largest eigenvalue of $A(G)$, and use $\mu_i(G)$ to denote the $i$-th smallest eigenvalue of $L(G)$. The second smallest eigenvalue of $L(G)$, $\mu_2(G)$, is known as the \textbf{algebraic connectivity} of $G$. Similarly to connectivity, a graph with a sufficiently high algebraic connectivity is also both rigid and globally rigid.
\begin{theorem}[Cioab\u{a}, Dewar and Gu~\cite{CDG21}]
\label{specrigid}
Let $G$ be a graph with minimum degree $\delta\ge 6$.
\begin{enumerate}[topsep=1pt,itemsep=0pt,leftmargin=9mm]
    \item If $\mu_{2}(G)> 2+ \frac{1}{\delta -1}$, then $G$ is rigid in $\mathbb{R}^2$.
    \item If $\mu_{2}(G)> 2+ \frac{2}{\delta -1}$, then $G$ is globally rigid in $\mathbb{R}^2$.
\end{enumerate}
\end{theorem}

In this paper, we investigate the rigidity properties of Ramanujan graphs. For $k\geq 3$, a connected $k$-regular graph $G$ is called a \textbf{Ramanujan graph} if $|\lambda_i(G)|\leq 2\sqrt{k-1}$ whenever $\lambda_i (G)\neq \pm k$ for every $1 \leq i \leq n$. These are sparse and highly connected graphs that are extremal with respect to their non-trivial adjacency matrix eigenvalues. They have been studied intensively over the last several decades (see \cite{AC,KS06,LPS88,morg92,Murty03,nilli}). As mentioned by Murty~\cite{Murty03}, the study of Ramanujan graphs involves diverse branches of mathematics such as combinatorics, number theory, representation theory, and algebra. By definition, if $G$ is a $k$-regular Ramanujan graph, then $\mu_2(G) \geq k - 2\sqrt{k-1}$. By Theorem~\ref{specrigid}, every Ramanujan graph of valency $k\geq 8$, is globally rigid in $\mathbb{R}^2$. In this paper, we study the cases of $k\leq 7$ and prove that all sufficiently large 6- and 7-regular Ramanujan graphs are globally rigid in $\mathbb{R}^2$.

\begin{theorem}\label{thm:ramanujan}
If $G$ is a $7$-regular Ramanujan graph with $n\geq 22$ vertices, then $G$ is globally rigid in $\mathbb{R}^2$. If $G$ is a $6$-regular Ramanujan graph with $n\geq 329$ vertices, then $G$ is globally rigid in $\mathbb{R}^2$.
\end{theorem} 

These bounds can be improved if $G$ is \textbf{vertex-transitive}; i.e., the automorphism group of $G$ acts transitively on the vertex set $V(G)$. In fact, rigidity and global rigidity of vertex-transitive graphs have been studied by Jackson, Servatius and Servatius~\cite{JSS07}. In particular, they proved that every $k$-regular vertex-transitive graph with $k\ge 6$ is globally rigid in $\mathbb{R}^2$. We therefore focus on $k=4, 5$ and prove the following theorem for $k$-regular Ramanujan graphs.

\begin{theorem}\label{thm:transitive}
Every vertex-transitive Ramanujan graph with degree at least 5 is globally rigid,
except the rigid but not globally rigid graph depicted in Figure \ref{fig:VertexTransitiveSpecialCase}. Every vertex-transitive Ramanujan graph with degree 4 that either has at least 53 vertices or is bipartite, is globally rigid.
\end{theorem}

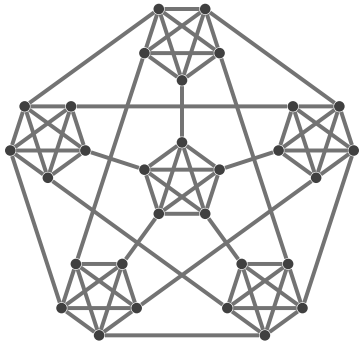
\begin{figure}[ht]
    \centering
    \begin{tikzpicture}[scale=0.75]
        \coordinate (o) at (0,0);
        \node[vertex] (01) at (18:0.7) {};
        \node[vertex] (02) at (90:0.7) {};
        \node[vertex] (03) at (162:0.7) {};
        \node[vertex] (04) at (234:0.7) {};
        \node[vertex] (05) at (306:0.7) {};
        \draw[edge] (01)edge(02) (01)edge(03) (01)edge(04) (01)edge(05) (02)edge(03) (02)edge(04) (02)edge(05) (03)edge(04) (03)edge(05) (04)edge(05);
        \foreach \x in {1,2,3,4,5}
        {
            \begin{scope}[rotate around={\x*72+18:(o)}]
                \begin{scope}[xshift=2.5cm]
                    \node[vertex] (\x1) at (0:-0.7) {};
                    \node[vertex] (\x2) at (72:-0.7) {};
                    \node[vertex] (\x3) at (144:-0.7) {};
                    \node[vertex] (\x4) at (-144:-0.7) {};
                    \node[vertex] (\x5) at (-72:-0.7) {};
                \end{scope}
                \draw[edge] (\x1)edge(\x2) (\x1)edge(\x3) (\x1)edge(\x4) (\x1)edge(\x5) (\x2)edge(\x3) (\x2)edge(\x4) (\x2)edge(\x5) (\x3)edge(\x4) (\x3)edge(\x5) (\x4)edge(\x5);
            \end{scope}
        }
        \draw[edge] (14)edge(23) (24)edge(33) (34)edge(43) (44)edge(53) (54)edge(13);
        \draw[edge] (01)edge(51) (02)edge(11) (03)edge(21) (04)edge(31) (05)edge(41);
        \draw[edge] (12)edge(45) (22)edge(55) (32)edge(15) (42)edge(25) (52)edge(35);
    \end{tikzpicture}
    \caption{The single special case to Theorem \ref{thm:transitive}. 
    The graph is not globally rigid (see Theorem~\ref{t:jjs07}),
    however it is rigid.
    To see that the graph is indeed rigid, note that if we delete a path of length 3 from each copy of $K_5$ contained in the graph,
    we obtain a $(2,3)$-tight graph.}
    \label{fig:VertexTransitiveSpecialCase}
\end{figure}

A specific class of $(p+1)$-regular Ramanujan graphs, denoted by $X^{p,q}$ for primes $p, q$ such that $p\equiv q \equiv 1 \pmod 4$, was constructed by Lubotzky, Phillips and Sarnak \cite{LPS88}. 
Generalizing this family of graphs, Morgenstern~\cite{morg92} constructed $(p+1)$-regular Ramanujan graphs for all prime powers $p$. Servatius~\cite{Serv} asked whether the Ramanujan graphs $X^{5,q}$ are rigid in $\mathbb{R}^2$. Since every graph $X^{5,q}$ is $6$-regular and vertex-transitive, the theorem of Jackson, Servatius and Servatius~\cite{JSS07} (see Theorem~\ref{t:jjs07} in this paper) actually implies an affirmative answer to this open question; that is, the Ramanujan graphs $X^{5,q}$ are globally rigid in $\mathbb{R}^2$. Their theorem also implies the global rigidity of Morgenstern's Ramanujan graphs for $p\ge 5$, however, not for $p=3,4$. Theorem~\ref{thm:transitive} fills the gap and implies that the Ramanujan graphs constructed by Morgenstern are globally rigid in $\mathbb{R}^2$ for $p=3,4$.

In Section~\ref{sec: proofs}, we present the proofs of the main theorems on global rigidity of Ramanujan graphs in $\mathbb{R}^2$. In Section~\ref{sec: connandst}, we list useful results on edge connectivity and edge-disjoint spanning trees, as well as some properties of Ramanujan graphs, which will be used for other types of framework rigidity in Section~\ref{sec: others}. These include body-bar (global) rigidity, body-hinge (global) rigidity, and rigidity on surfaces of revolution. In Section~\ref{sec:comp}, we use computational methods to improve some previous results in the paper.
Finally, we conclude the paper with some open questions in Section~\ref{sec:end}.

\section{Proofs of Theorems~\ref{thm:ramanujan} and \ref{thm:transitive}}
\label{sec: proofs}

\subsection{General case}

We use the following theorem in our proofs.
\begin{theorem}[Jackson and Jord\'an~\cite{JaJo09}]\label{thm:JJ}
A simple graph $G$ is globally rigid in $\mathbb{R}^2$ if $G$ is $6$-edge-connected, $G-u$ is $4$-edge-connected for every vertex $u$ and $G-\{v, w\}$ is $2$-edge-connected for any vertices $v,w\in V(G)$.
\end{theorem}

\begin{proof}[Proof of Theorem~\ref{thm:ramanujan}]
For the first part, we prove that if $G$ is a $7$-regular Ramanujan graph with $n\geq 22$ vertices, then $G$ is $6$-edge-connected, $G-\{u\}$ is $4$-edge-connected for any vertex $u$ and $G-\{v,w\}$ is $2$-edge-connected for any vertices $v$ and $w$. Then the result follows by Theorem~\ref{thm:JJ}. 

Because $7-\lambda_2\geq 7-2\sqrt{6}>2.1$, $G$ is $7$-edge-connected (see Theorem~\ref{thm:ecn} or \cite[Theorem 1.3]{Cio10} or \cite[Theorem 4.3]{KS06}). We show that for any subset of vertices $S$ with $2\leq |S|\leq n/2$, $e(S,V\setminus S)\geq 11$. If $2\leq |S|\leq 6$, then $e(S,V\setminus S)\geq |S|(7-|S|+1)\geq 12$. If $7\leq |S|\leq n/2$, then we use the spectral bound $e(S,V\setminus S)\geq \frac{(7-\lambda_2)|S||V\setminus S|}{n}$ (see \cite[Lemma 1.2]{Cio10}) and obtain that $e(S,V\setminus S)\geq \frac{(7-2\sqrt{6})7(n-7)}{n}>10$ for $n>\frac{49(7-2\sqrt{6})}{39-14\sqrt{6}}\approx 21.87$. Hence, $e(S,V\setminus S)\geq 11$ when $n\geq 22$ as claimed.
Because $G$ is $7$-regular, this implies that for any vertex $u$ of $G$, the graph $G-\{u\}$ is $4$-edge-connected.

We now prove that $G$ is $4$-connected. This implies that $G-\{v, w\}$ is $2$-connected and, therefore, $2$-edge-connected. From Fiedler \cite{Fiedler}, we know that the vertex-connectivity of $G$ is at least $7-\lambda_2\geq 7-2\sqrt{6}>2$, implying that $G$ is $3$-connected.
By contradiction, assume that $G$ has a subset $T$ of three vertices such that $G-T$ is disconnected. Let $A\cup B=V(G)\setminus T$ be a partition of $V(G)\setminus T$ such that there are no edges between $A$ and $B$. Denote $a=|A|$ and $b=|B|$. The Expander Mixing Lemma (see \cite{AC} or \cite[Theorem 2.11]{KS06}) implies that $7ab/n\leq 2\sqrt{6}\sqrt{ab(1-a/n)(1-b/n)}$. A straightforward calculation gives that $n\geq 25ab/72$. Because the neighborhood of $A$ (or $B$) is $T$, it follows that $a,b\geq 5$. Since $a+b=n-3$, we get that $n\geq 25\cdot 5(n-8)/72$ which implies that $n\leq 1000/53<20$, contradiction. Hence, the vertex-connectivity of $G$ is at least $4$ as claimed. This finishes our proof of the first part.

We use a similar strategy for the second part and show that if $G$ is a $6$-regular Ramanujan graph with $n\geq 329$ vertices, then $G$ is $6$-edge-connected, $G-\{u\}$ is $4$-edge-connected and $G-\{v,w\}$ is $2$-edge-connected for any vertices $v$ and $w$. We first show that for any subset of vertices $S$ with $2\leq |S|\leq n/2$, $e(S,V\setminus S)\geq 10$. If $2\leq |S|\leq 5$, it is easy to see that $e(S,V\setminus S)\geq |S|(6-|S|+1)\geq 10$. If $6\leq |S|\leq n/2$, then we use again the spectral bound $e(S,V\setminus S)\geq \frac{(6-\lambda_2)|S|(n-|S|)}{n}$ and get that $e(S,V\setminus S)\geq  \frac{(6-2\sqrt{5})6(n-6)}{n}>9$ for $n>\frac{72-24\sqrt{5}}{9-4\sqrt{5}}\approx 328.99$. Hence, $e(S,V\setminus S)\geq 10$ when $n\geq 329$ as claimed. This means that $G$ is $6$-edge-connected and $G-\{u\}$ is $4$-edge-connected for any vertex $u$ of $G$.

We now prove that $G$ is $4$-connected, which will in turn imply that $G-\{v,w\}$ is $2$-connected and, therefore, $2$-edge-connected.
From Fiedler \cite{Fiedler}, the vertex-connectivity of $G$ is at least $6-\lambda_2\geq 6-2\sqrt{5}>1.52$, i.e.~$G$ is 2-connected.
By contradiction, assume that $G$ has a subset $T$ of two or three vertices such that $G-T$ is disconnected. Let $A\cup B=V(G)\setminus T$ be a partition of $V(G)\setminus T$ such that there are no edges between $A$ and $B$. Denote $a=|A|$ and $b=|B|$. Using again the Expander Mixing Lemma (see \cite{AC} or \cite[Theorem 2.11]{KS06}), we obtain that $6ab/n\leq \lambda \sqrt{ab(1-a/n)(1-b/n)}\leq 2\sqrt{5}  \sqrt{ab(1-a/n)(1-b/n)}$. It follows that $n\geq 4ab/15$. Because the neighborhood of $A$ (or $B$) is $T$, we get that $a,b\geq 4$, and hence the minimum possible value of $ab$ is $4(n-7)$. Therefore, $n\geq 4ab/15\geq 4\cdot 4(n-7)/15$ which implies that $n\leq 112$, contradiction. This finishes the proof of the second part.
\end{proof}

\subsection{Vertex-transitive case}

The global rigidity of vertex-transitive graphs was previously characterized by Jackson, Servatius and Servatius~\cite{JSS07} in the following theorem.

\begin{theorem}[Jackson, Servatius and Servatius~\cite{JSS07}]\label{t:jjs07}
Let $G=(V,E)$ be a connected vertex-transitive graph of degree $k \geq 2$. Then $G$ is globally rigid in $\mathbb{R}^2$ if and only if one of the following holds.
\begin{enumerate}[topsep=1pt,itemsep=0pt,leftmargin=9mm]
    \item $k=2$ and $|V| \leq 3$.
    \item $k=3$ and $|V| \leq 4$.
    \item $k=4$, and either the maximal clique size is at most 3 or $|V| \leq 11$.
    \item $k=5$, and either the maximal clique size is at most 4 or $|V| \leq 28$.
    \item $k \geq 6$.
\end{enumerate}
\end{theorem}

It follows from Theorem \ref{t:jjs07} that any vertex-transitive bipartite graph of degree $k \geq 4$ is globally rigid in $\mathbb{R}^2$. It is also easy to construct connected $k$-regular vertex-transitive non-bipartite graphs for $k \in \{4,5\}$, that are not globally rigid in $\mathbb{R}^2$. To do so, take any $k$-regular graph, replace every vertex with a copy of $K_k$ and then share the edges out evenly amongst the new cliques. Any such graph will, however, have a relatively low algebraic connectivity.

\begin{lemma}\label{l:cliqreplace}
Let $k \geq 3$ and $G$ be a connected $k$-regular graph where every vertex is contained in a clique of size $k$. If $H$ is the multigraph formed from contracting every clique of size $k$ to a point, then $H$ is a well-defined $k$-regular multigraph with $\mu_2(G) \leq \mu_2(H)/k$.
\end{lemma}
\begin{proof}
Let $C_1,\ldots,C_n$ be the cliques of size $k$ of $G$ (and hence also the vertices of $H$). We note that no two cliques of size $k$ share a vertex by our degree bound, hence $H$ is well-defined. Let $x \in \mathbb{R}^{V(H)}$ be a unit eigenvector of $\mu_2(H)$, and define $\tilde{x}^{V(G)}$ to be the vector where for each $v \in C_i$, we set $\tilde{x}(v) = x(C_i)$. We immediately compute that $\tilde{x}^T \tilde{x} = k$ and
\begin{align*}
	\tilde{x}^T L(G) \tilde{x} = x^T L(H) x = \mu_2(H).
\end{align*}
The result now follows as
\begin{align*}
	\mu_2(G) = \min_{u \in [1 ~ \ldots ~ 1]^\perp} \frac{u^T L(G) u}{u^T u} \leq \frac{\tilde{x}^T L(G) \tilde{x}}{\tilde{x}^T \tilde{x}} = \frac{\mu_2(H)}{k}.
\end{align*}	
\end{proof}

To prove Theorem \ref{thm:transitive}, we also require the following three technical results.

\begin{theorem}[Nilli~\cite{nilli}]\label{t:nilli}
Let $G$ be a $k$-regular (multi)graph with diameter $m>1$. Then
	\begin{align*}
		\mu_2(G) \leq k - 2 \sqrt{k-1} + \frac{2 \sqrt{k-1} -1}{\left\lfloor m/2\right\rfloor}.
	\end{align*}
\end{theorem}

\begin{lemma}\label{l:vtspec}
Let $k \in \{4,5\}$ and $G$ be a connected vertex-transitive graph with degree $k$ and diameter $m>1$.	If 
	\begin{align*}
		\mu_2(G) > 1 -  \frac{2\sqrt{k-1}}{k} + \frac{2 \sqrt{k-1} -1}{k\left\lfloor m/2\right\rfloor},
	\end{align*}
then either $G$ is globally rigid in $\mathbb{R}^2$, or $G$ is one of the graphs depicted in Figures~\ref{fig:VertexTransitiveSpecialCase} and \ref{fig:secondSpecialCase}.
\end{lemma}
\begin{proof}
Suppose $G$ is not globally rigid in $\mathbb{R}^2$ and is not one of the graphs depicted in Figures~\ref{fig:VertexTransitiveSpecialCase} and~\ref{fig:secondSpecialCase}. It is immediate that $G$ is not a complete graph.
By Theorem \ref{t:jjs07}, $G$ contains a clique of size $k$ and has at least $n$ vertices, where $n = 12$ if $k=4$ and $n=30$ if $k=5$. As $G$ is vertex-transitive, every vertex of $G$ lies in a clique of size $k$. Define $H$ to be the $k$-regular vertex-transitive multigraph formed from contracting each clique of size $k$ to a point. If $H$ has a diameter of at least 2, then $\mu_2(G) \leq 1 -  \frac{2\sqrt{k-1}}{k} + \frac{2 \sqrt{k-1} -1}{k\left\lfloor m/2\right\rfloor}$ by Lemma~\ref{l:cliqreplace} and Theorem~\ref{t:nilli}.
	
Suppose that $H$ has diameter one. If $k=5$ then $H$ must have at least 6 vertices so that $G$ has at least 30 vertices. This implies that $H = K_6$ and $G$ is the graph depicted in Figure~\ref{fig:VertexTransitiveSpecialCase}, which contradicts our original assumption. If $k=4$ then $H$ is either $K_5$ or $2K_3$ (the multigraph formed from $K_3$ by doubling each edge). This implies that $G$ is one of the graphs depicted in Figure~\ref{fig:secondSpecialCase}, which contradicts our original assumption.
\end{proof}

\begin{figure}[ht]
    \centering
    \begin{tikzpicture}[scale=0.6]
        \coordinate (o) at (0,0);
        \foreach \x in {1,2,3}
        {
            \begin{scope}[rotate around={\x*120-30:(o)}]
                \begin{scope}[xshift=2cm]
                    \node[vertex] (\x1) at (-45:1) {};
                    \node[vertex] (\x2) at (45:1) {};
                    \node[vertex] (\x3) at (135:1) {};
                    \node[vertex] (\x4) at (-135:1) {};
                \end{scope}
                \draw[edge] (\x1)edge(\x2) (\x1)edge(\x3) (\x1)edge(\x4) (\x2)edge(\x3) (\x2)edge(\x4) (\x3)edge(\x4);
            \end{scope}
        }
        \draw[edge] (12)edge(21) (22)edge(31) (32)edge(11);
        \draw[edge] (13)edge(24) (23)edge(34) (33)edge(14);
    \end{tikzpicture}
    \qquad
    \begin{tikzpicture}[scale=0.6,rotate=90]
        \coordinate (o) at (0,0);
        \foreach \x in {1,2,3,4,5}
        {
            \begin{scope}[rotate around={\x*72:(o)}]
                \begin{scope}[xshift=2cm]
                    \node[vertex] (\x1) at (-45:0.8) {};
                    \node[vertex] (\x2) at (45:0.8) {};
                    \node[vertex] (\x3) at (135:0.8) {};
                    \node[vertex] (\x4) at (-135:0.8) {};
                \end{scope}
                \draw[edge] (\x1)edge(\x2) (\x1)edge(\x3) (\x1)edge(\x4) (\x2)edge(\x3) (\x2)edge(\x4) (\x3)edge(\x4);
            \end{scope}
        }
        \draw[edge] (12)edge(21) (22)edge(31) (32)edge(41) (42)edge(51) (52)edge(11);
        \draw[edge] (13)edge(34) (23)edge(44) (33)edge(54) (43)edge(14) (53)edge(24);
    \end{tikzpicture}
    \caption{Two 4-regular vertex-transitive graphs that are not globally rigid in $\mathbb{R}^2$.
    The graph on the left is rigid in $\mathbb{R}^2$, but the graph on the right is not.}
    \label{fig:secondSpecialCase}
\end{figure}
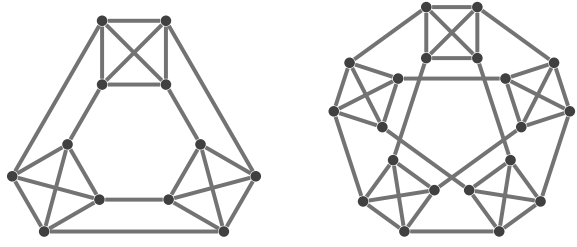

\begin{theorem}[\cite{BanIto73, Dam73, HoffSing60}]\label{t:moore}
	If $G=(V,E)$ is a $k$-regular graph with diameter $m$, then
	\begin{align*}
		|V| \leq 1 + k \sum_{i=0}^{m-1} (k-1)^i.
	\end{align*}
	Furthermore, this bound is strict if $k \notin \{2,3,7,57\}$.
\end{theorem}

We are now ready to prove Theorem~\ref{thm:transitive}.

\begin{proof}[Proof of Theorem~\ref{thm:transitive}]
Let $G$ be a $k$-regular vertex-transitive graph. If $G$ is complete, then we are done. Hence, we may assume the diameter of $G$ is at least 2. If $k \geq 6$ then $G$ is globally rigid in $\mathbb{R}^2$ by Theorem \ref{t:jjs07}. Thus, we may suppose $k \in \{4,5\}$. If $G$ is bipartite then $G$ is globally rigid in $\mathbb{R}^2$ by Theorem \ref{t:jjs07}, and so we may also suppose $G$ is not bipartite. If $k=5$ then
\begin{align*}
	\mu_2(G) \geq 1 > 1 - \frac{2 \sqrt{4}}{5} + \frac{2 \sqrt{4} -1}{5} \geq 1 - \frac{2 \sqrt{4}}{5} + \frac{2 \sqrt{4} -1}{5\left\lfloor m/2\right\rfloor},
\end{align*}
and so either $G$ is globally rigid in $\mathbb{R}^2$ by Lemma \ref{l:vtspec}, or $G$ is the graph in Figure \ref{fig:VertexTransitiveSpecialCase}. If $k=4$ and $|V| \geq 53$, then $G$ has diameter at least 4 by Theorem \ref{t:moore}. Since $|V| > 20$, $G$ is neither of the graphs in Figure \ref{fig:secondSpecialCase}. It now follows that
\begin{align*}
	\mu_2(G) \geq 4- 2 \sqrt{3} > 1 - \frac{2 \sqrt{3}}{4} + \frac{2 \sqrt{3} -1}{8} > 1 - \frac{2 \sqrt{3}}{4} + \frac{2 \sqrt{3} -1}{4\left\lfloor m/2\right\rfloor},
\end{align*}
and thus $G$ is globally rigid in $\mathbb{R}^2$ by Lemma \ref{l:vtspec}.
\end{proof}

\section{Edge connectivity and edge-disjoint spanning trees}
\label{sec: connandst}
Edge connectivity and edge-disjoint spanning trees are closely related to various types of graph rigidity that we shall explore in later sections. In this section, we list several useful results. Throughout the section we always assume $k, \ell, s, t$ are positive integers. The following result is the well-known spanning tree packing theorem.

\begin{theorem}[Nash-Williams~\cite{Nash61} and Tutte~\cite{Tutt61}]\label{NaTu}
A connected {\em(}multi{\em)}graph $G$ has $k$ edge-disjoint spanning trees if and only if for any $X\subseteq E(G), |X|\ge k(c(G-X)-1)$, where $c(G-X)$ denotes the number of connected components of $G-X$.
\end{theorem}

The above spanning tree packing theorem implies that if $G$ is $2k$-edge-connected, then $G-e$ has $k$ edge-disjoint spanning trees for every edge $e$ of $G$. This can be improved by using the following parameter.
Define the \textbf{strength} $\eta(G)$ for a connected (multi)graph $G$ by $$\eta(G)=\min\frac{|X|}{c(G-X)-1},$$
where the minimum is taken over all edge subsets $X$ such that $G-X$ is disconnected. This parameter was first introduced for graphs by Gusfield~\cite{Gusf83}, was then extended to matroids by Cunningham~\cite{Cunn85}, and has been intensively studied in \cite{CGHL92}. The above spanning tree packing theorem by Nash-Williams~\cite{Nash61} and Tutte~\cite{Tutt61} indicates that a connected graph $G$ has $k$ edge-disjoint spanning trees if and only if $\eta(G)\ge k$. In other words, the maximum number of edge-disjoint spanning trees in $G$ is $\lfloor\eta(G)\rfloor$. Thus, $\eta(G)$ is also referred to as the \textbf{fractional spanning tree packing number} of $G$. The spanning tree packing theorem implies the following result.

\begin{corollary}\label{thm:gminus}
Let $G$ be a connected {\em(}multi{\em)}graph. Then $G-e$ has $k$ edge-disjoint spanning trees for every $e\in E(G)$ if and only if $\eta(G)>k$.
\end{corollary}

Cioab\u{a} and Wong \cite{CiWo12} initiated the investigation of the number of edge-disjoint spanning trees from eigenvalues and posed a conjecture that if $G$ is a $d$-regular graph with $\lambda_2(G) <d -\frac{2k-1}{d+1}$, then $G$ contains $k$ edge-disjoint spanning trees, where $d\geq 2k\geq 4$. The conjecture was completely settled in \cite{LHGL14}, and actually it was proved that if $G$ is a graph with minimum degree $\delta \ge 2k$ and $\mu_2(G)>\frac{2k-1}{\delta+1}$, then $G$ contains $k$ edge-disjoint spanning trees. This result was extended from simple graphs to multigraphs in \cite{Gu16} and from spanning tree packing to a fractional version in \cite{HGLL16}. We outline these results below.

\begin{theorem}[Liu et al.~\cite{LHGL14}]\label{thm:LHGL}
Let $G$ be a graph with minimum degree $\delta\geq 2k$. If $\mu_{2}(G)>\frac{2k-1}{\delta+1}$ (in particular, if $\lambda_2(G)< \delta-\frac{2k-1}{\delta+1}$), then $G$ has at least $k$ edge-disjoint spanning trees.
\end{theorem}

\begin{theorem}[Gu~\cite{Gu16}]\label{thm:mgac}
Let $G$ be a multigraph with multiplicity $m$ and minimum degree $\delta \geq 2k$, and define $\ell := \max \left\{ \left\lceil (\delta +1)/m\right\rceil, 2 \right\}$. If $\mu_2(G) > \frac{2k-1}{\ell}$, then $G$ contains $k$ edge-disjoint spanning trees.
\end{theorem}

\begin{theorem}[Hong et al.~\cite{HGLL16}]\label{thm:frac}
Let $G$ be a graph with $\delta\ge 2s/t$.
If $\mu_2(G)>\frac{2s-1}{t(\delta +1)}$, then $\eta(G)\ge s/t$.
\end{theorem}

The following spectral conditions for edge connectivity were provided in \cite{Cio10} for regular graphs, and similar results were proved for general graphs by \cite{GLLY12, LiHL13}.
\begin{theorem}[Cioab\u{a}~\cite{Cio10}]\label{thm:ecn}
Let $G$ be a $k$-regular graph with $n$ vertices and $k\ge \ell \ge 2$.
If $\lambda_2(G) \le k -\frac{(\ell -1)n}{(k+1)(n-k-1)}$, then $G$ is $\ell$-edge-connected. In particular, if $\lambda_2(G)< k -\frac{2(\ell -1)}{k+1}$, then $G$ is $\ell$-edge-connected.
\end{theorem}

\begin{remark}
Note that for $n < 2k +2$, every $k$-regular graph $G$ is $k$-edge-connected. To see this, suppose that there is a nonempty proper vertex subset $A\subset V(G)$ such that there are less than $k$ edges between $A$ and its complement $\overline{A}$. By counting the degree sum in $A$, we obtain that $k|A|< |A|(|A|-1)+k$, which implies that $|A| > k$, and so $|A|\geq k+1$. 
Similarly, $|\overline{A}|\geq k+1$, hence $n = |A| + |\overline{A}|\geq 2k+2$, contrary to $n < 2k+2$. Thus, in Theorem~\ref{thm:ecn}, we can always assume $n\ge 2k+2$.
\end{remark}

\begin{theorem}[Cioab\u{a} and Gu~\cite{CiGu16}]\label{thm:2-conn}
For any connected $k$-regular graph $G$ with $k\ge 3$, if 
$$\lambda_2 (G) < \left\{ \begin{array}{ll}
        \frac{k-2 + \sqrt{k^2 +12}}{2}, & \textrm{if $k$ is even},\\ 
        \frac{k-2 + \sqrt{k^2 +8}}{2}, & \textrm{if $k$ is odd},
        \end{array} \right.
$$
then $G$ is 2-connected.
\end{theorem}

From the above results, we can obtain the following connectivity properties for Ramanujan graphs.
\begin{proposition}\label{pro:edgeconn}
Let $G$ be a $k$-regular Ramanujan graphs with $n$ vertices.
\begin{enumerate}[topsep=1pt,itemsep=0pt,leftmargin=9mm]
    \item\label{it:edgeconn:k6} If $k\ge 6$, then $G$ is $k$-edge-connected.
    \item\label{it:edgeconn:k5} If $k= 5$, then $G$ is 4-edge-connected.
    \item\label{it:edgeconn:k4} If $k=4$ and $n\ge 20$ (or $n\le 9$), then $G$ is 4-edge-connected.
    \item\label{it:edgeconn:k3} If $k\ge 4$, then $G$ is $2$-connected.
\end{enumerate}
\end{proposition}
\begin{proof}
\ref{it:edgeconn:k6} It is not hard to check $2\sqrt{k-1} < k -\frac{2(k -1)}{k+1}$ when $k\ge 6$. Now it follows easily from Theorem~\ref{thm:ecn} with $\ell =k$.

\ref{it:edgeconn:k5} It is not hard to check $\lambda_2(G)\le 2\sqrt{k-1} \le k -\frac{(\ell -1)n}{(k+1)(n-k-1)}$ for $k=5,\ell =4$ and $n\ge 12$. By Theorem~\ref{thm:ecn} (for $n\ge 12$) and its remark (for $n<12$), $G$ is 4-edge-connected.

\ref{it:edgeconn:k4} For 4-regular Ramanujan graphs with $n\ge 20$ vertices, it is easy to check $\lambda_2 \le 2\sqrt{3} < 4 -\frac{(3 -1)n}{5(n-5)}$, and thus is 3-edge-connected by Theorem~\ref{thm:ecn}. However, we know that for $k$-regular graphs, if $k$ is even, then the edge connectivity is also even (see \cite[Lemma 3.1]{Cio10} for a proof). Thus $G$ is 4-edge-connected.
If $n\le 9$ then, by the remark of Theorem~\ref{thm:ecn}, $G$ is 4-edge-connected.

\ref{it:edgeconn:k3} This follows directly from Theorem~\ref{thm:2-conn}.
\end{proof}

We finish the section by briefly discussing some properties of the following graph operation. Denote by $tG$ the multigraph obtained from $G$ by replacing every edge with $t$ parallel edges. Conveniently, the fractional spanning tree packing number respects this ``scalar multiplication'' operation.
\begin{lemma}[{\cite[Lemma 1]{LaLa91}}]
\label{thm:eta}
$\eta(tG) = t\eta(G)$.
\end{lemma}

By combining Lemma \ref{thm:eta} with the results of this section, we obtain the following result.

\begin{lemma}\label{lem:tgst}
Let $G$ be a graph with minimum degree $\delta > 2s/t$.
If $\mu_2(G) > \frac{2s}{t(\delta+1)}$, then $tG -e$ has at least $s$ edge-disjoint spanning trees for every $e\in E(tG)$.
\end{lemma}
\begin{proof}
By Corollary~\ref{thm:gminus}, we need to show that $\eta(tG)>s$. 
Since $\mu_2(G) > \frac{2s}{t(\delta+1)}$, we can choose a sufficient small rational number $\varepsilon >0$ such that $\mu_2(G) >\frac{2(s+\varepsilon )}{t(\delta+1)}$. Let $s', t'$ be positive integers such that $\frac{s'}{t'} =\frac{s+\varepsilon}{t}$. Then $\mu_2(G) >\frac{2(s+\varepsilon )}{t(\delta+1)} =\frac{2s'}{t'(\delta+1)}$ and by Theorem~\ref{thm:frac}, we have $\eta(G)\ge \frac{s'}{t'} =\frac{s+\varepsilon}{t}$. By Lemma~\ref{thm:eta}, $\eta(tG) = t\eta(G)\ge s+\varepsilon >s$.
\end{proof}

\section{Sufficient conditions for other types of framework rigidity}
\label{sec: others}
In this section, we study different types of framework rigidity, including body-bar rigidity, body-hinge rigidity, and rigidity on surfaces of revolution. 

\subsection{Body-and-bar rigidity}
We begin by studying \textbf{body-and-bar frameworks} in $\mathbb{R}^d$, i.e., frameworks of $d$-dimension rigid bodies that are connected by fixed-length bars attached at points of their surfaces; see \cite{Tay84} for more details. Informally, we say a multigraph $G$ is \textbf{body-bar rigid in $\mathbb{R}^d$} if there exists a generic rigid body-bar framework in $\mathbb{R}^d$, and \textbf{body-bar globally rigid in $\mathbb{R}^d$} if there exists a generic globally rigid body-bar framework in $\mathbb{R}^d$. Since any two vertices connected by $\frac{d(d+1)}{2}$ edges can be considered to be the same rigid body, we make the assumption that the multiplicity of our graphs is less than $\frac{d(d+1)}{2}$. Instead of rigorous definitions, we may instead characterize these two combinatorial properties exactly by the following results.

\begin{theorem}[\cite{Tay84}]
A multigraph $G$ is body-bar rigid in $\mathbb{R}^d$ if and only if it contains $\frac{d(d+1)}{2}$ edge-disjoint spanning trees.
\end{theorem}

\begin{theorem}[\cite{ConnJorW13}]
A multigraph $G$ is body-bar globally rigid in $\mathbb{R}^d$ if and only if it is redundantly rigid in $\mathbb{R}^d$, i.e.,~$G-e$ is body-bar rigid in $\mathbb{R}^d$ for all $e \in E$.
\end{theorem}

For the following, we set $\ell := \max \left\{ \left\lceil (\delta +1)/m\right\rceil, 2 \right\}$, where $\delta$ is the minimum degree and $m$ is the multiplicity of the graph in question. We notice that if $\delta \geq d(d+1)$ and $m < \binom{d+1}{2} = \frac{d(d+1)}{2}$, then $\ell \geq 3$.

\begin{corollary}\label{cor:bbramrig}
Let $G$ be a $k$-regular Ramanujan multigraph with $k \geq d(d+1)$ for $d \geq 2$ and multiplicity $m < \frac{d(d+1)}{2}$. Then $G$ is body-bar rigid in $\mathbb{R}^d$.
\end{corollary}

\begin{proof}
First suppose $d \geq 3$ and fix $D = d(d+1)$. As mentioned previously, $\ell \geq 3$. By Theorem~\ref{thm:mgac}, it suffices to show that
\begin{align*}
    D - 2 \sqrt{D-1}  > \frac{2\left(\frac{d(d+1)}{2} \right)-1}{3} = \frac{D - 1}{3},
\end{align*}
as then $G$ will contain $\frac{d(d+1)}{2}$ edge-disjoint spanning trees. By rearranging we obtain the quadratic inequality $4D^2 -32D + 37 >0$ which holds for all $D \geq 12$, hence $G$ is body-bar rigid in $\mathbb{R}^d$.
    
Now suppose $d=2$. Since $m\leq 2$ we have $\ell \geq 4$. As
\begin{align*}
    \mu_2(G) > 6 - 2 \sqrt{5}  > \frac{5}{4} = \frac{2.3 - 1}{\ell},
\end{align*}
the graph $G$ is body-bar rigid in $\mathbb{R}^2$ by Theorem \ref{thm:mgac}.
\end{proof}

This result can be seen to be the best possible result, since any $k$-regular graph for $k < d(d+1)$ with $n > d(d+1)$ vertices has at most $\left(\frac{d(d+1)-1}{2} \right)n < \left(\frac{d(d+1)}{2} \right)(n-1)$ edges, and thus cannot contain $\frac{d(d+1)}{2}$ edge-disjoint spanning trees.

We finish the section by characterizing the body-bar globally rigid Ramanujan graphs.

\begin{corollary}\label{cor:bbramgr}
Let $G$ be a $k$-regular Ramanujan multigraph with $k \geq d(d+1) +2$ for $d \geq 2$ and multiplicity $m < \frac{d(d+1)}{2}$. Then $G$ is body-bar globally rigid in $\mathbb{R}^d$.
\end{corollary}

\begin{proof}
As mentioned previously, $\max \left\{ \left\lceil (k +1)/m\right\rceil, 2 \right\} \geq 3$. By Theorem \ref{thm:mgac}, it suffices to show that
\begin{align*}
    D+2 - 2\sqrt{D+1}> \frac{2\left(\frac{d(d+1)}{2} + 1 \right)-1}{3} =\frac{D + 1}{3},
\end{align*}
as then $G$ will contain $\frac{d(d+1)}{2}+1$ edge-disjoint spanning trees. By rearranging we obtain the quadratic inequality $4D^2 - 16D - 11 >0$ which holds for $D \geq 6$, hence $G$ is body-bar globally rigid in $\mathbb{R}^d$.
\end{proof}

This result can likewise be seen to be the best possible result, since any $k$-regular graph for $k < d(d+1)+2$ with $n> d(d+1)+2$ vertices has at most $\left(\frac{d(d+1)+1}{2} \right)n < \left(\frac{d(d+1)}{2} + 1 \right)(n-1)$ edges, and hence cannot contain $\frac{d(d+1)}{2} + 1$ edge-disjoint spanning trees.

\subsection{Body-and-hinge rigidity}

In this subsection, we study \textbf{body-and-hinge frameworks} in $\mathbb{R}^d$, i.e., frameworks of $d$-dimension rigid bodies that are connected by $(d-1)$-dimensional hinges that on their surfaces; see \cite{JaJo10} for more details. Unlike body-and-bar frameworks, we restrict ourselves to simple graphs, as two bodies connected by two hinges are essentially the same body.
Informally, we say a graph $G$ is \textbf{body-hinge rigid in $\mathbb{R}^d$} if there exists a generic rigid body-hinge framework in $\mathbb{R}^d$, and \textbf{body-hinge globally rigid in $\mathbb{R}^d$} if there exists a generic globally rigid body-hinge framework in $\mathbb{R}^d$. Instead of rigorous definitions, we may instead characterize these two combinatorial properties exactly by the following results. We recall that for any graph $G$ and any $k \in \mathbb{N}$, the multigraph $kG$ is formed from $G$ by replacing every edge with $k$ parallel copies.

\begin{theorem}[\cite{TayWhiteley82}]\label{thm:bhr}
A graph $G$ is body-hinge rigid in $\mathbb{R}^d$ if and only if $(\binom{d+1}{2}-1)G$ contains $\binom{d+1}{2}$ edge-disjoint spanning trees.
\end{theorem}

\begin{theorem}[\cite{JordKirTan16}]\label{thm:bhgr}
A graph $G$ is body-hinge globally rigid in $\mathbb{R}^d$ if and only if either (i) $d=2$ and $G$ is 3-edge-connected, or (ii) $d \geq 3$ and $(\binom{d+1}{2}-1)G -e$ contains $\binom{d+1}{2}$ edge-disjoint spanning trees for all $e \in E$.
\end{theorem}

We observe an interesting quirk of body-hinge frameworks that follows immediately from Theorems \ref{thm:bhr} and \ref{thm:bhgr}.

\begin{corollary}
If $G$ is a graph that is body-hinge $($globally$)$ rigid in $\mathbb{R}^d$, then, for all $D \geq d$, $G$ is body-hinge $($globally$)$ rigid in $\mathbb{R}^D$.
\end{corollary}

The following result gifts us sufficient spectral conditions for body-hinge rigidity and global rigidity.

\begin{theorem}\label{thm:bh}
Let $G$ be a graph with minimal degree $\delta \geq 3$.
\begin{enumerate}[topsep=1pt,itemsep=0pt,leftmargin=9mm]
    \item\label{it:bh:bhr} If
            \begin{align*}
                \mu_2(G) > \frac{1}{\delta+1} \left( 2+ \frac{1}{\binom{d+1}{2}-1} \right),
            \end{align*}
            then $G$ is body-hinge rigid in $\mathbb{R}^d$ for $d\ge 2$.
    \item\label{it:bh:bhgr} If
            \begin{align*}
                \mu_2(G) > \frac{1}{\delta+1} \left( 2+ \frac{2}{\binom{d+1}{2}-1} \right),
            \end{align*}
            then $G$ is body-hinge globally rigid in $\mathbb{R}^d$ for $d\ge 3$.
\end{enumerate}
\end{theorem}
\begin{proof}
\ref{it:bh:bhr} Set $D := \binom{d+1}{2}$. Since $D \geq 3$,
\begin{align*}
\delta((D-1)G) = (D-1)\delta \geq 3 (D-1) \geq 2D.
\end{align*}
As $m = D-1$, we see that
\begin{align*}
\ell = \max \left\{ \frac{ \delta((D-1)G) +1}{m},2 \right\} = \max \left\{  \delta + \frac{1}{D-1},2 \right\} = \delta +1.
\end{align*}
Finally, as $\mu_2((D-1)G) = (D-1) \mu_2(G)$, we have that
\begin{align*}
\mu_2((D-1)G) > (D-1) \frac{1}{\delta+1} \left( 2+ \frac{1}{D-1} \right) = \frac{2D-1}{\delta+1} = \frac{2D-1}{\ell}.
\end{align*}

By Theorem~\ref{thm:mgac}, $(D-1)G$ has at least $D$ edge-disjoint spanning trees, and thus $G$ is body-hinge rigid in $\mathbb{R}^d$ by Theorem~\ref{thm:bhr}.

\ref{it:bh:bhgr} We have
\begin{align*}
\mu_2(G) > \frac{1}{\delta+1} \left( 2+ \frac{2}{\binom{d+1}{2}-1}\right)
=\frac{1}{\delta+1} \left( 2+ \frac{2}{D-1} \right) = \frac{2D}{(D-1)(\delta+1)}.
\end{align*}
By Lemma~\ref{lem:tgst}, $(D-1)G-e$ has at least $D$ edge-disjoint spanning trees for every edge $e$. Thus $G$ is body-hinge globally rigid in $\mathbb{R}^d$ by Theorem~\ref{thm:bhgr}.
\end{proof}

By combining Proposition~\ref{pro:edgeconn} and the above theorems, we obtain the following results. 

\begin{corollary}\label{cor:bhram}
Let $G$ be a $k$-regular Ramanujan graph with $n$ vertices.
\begin{enumerate}[topsep=1pt,itemsep=0pt,leftmargin=9mm]
    \item If $k\geq 4$, then $G$ is body-hinge rigid in $\mathbb{R}^2$.
    \item If $k\geq 5$, or if $k= 4$ and $n\ge 20$ $($or $n\le 9)$, then $G$ is body-hinge globally rigid in $\mathbb{R}^2$.
    \item If $k\geq 4$, then $G$ is body-hinge globally rigid in $\mathbb{R}^d$ for $d\geq 3$.
\end{enumerate}
\end{corollary}

Note that Corollary~\ref{cor:bhram} cannot be extended to cubic Ramanujan graphs in general, as there exist cubic Ramanujan graphs that are not body-hinge rigid in $\mathbb{R}^d$ for any $d\geq 2$. The cubic Ramanujan graph on the left in Figure~\ref{fig:cubicRamanujan} was originally constructed in \cite{Cio10}, and its largest absolute eigenvalue smaller than 3 equals the largest root of the equation $x^3-7x-2=0$ (roughly about $2.7786$), which is less than $2\sqrt{2}$. 
However, since $G$ has edge-connectivity 1, $tG$ will have edge-connectivity $t$ for each positive integer $t$ and so will contain at most $t$ edge-disjoint spanning trees. By Theorems~\ref{thm:bhr} and \ref{thm:bhgr}, $G$ is not body-hinge rigid.

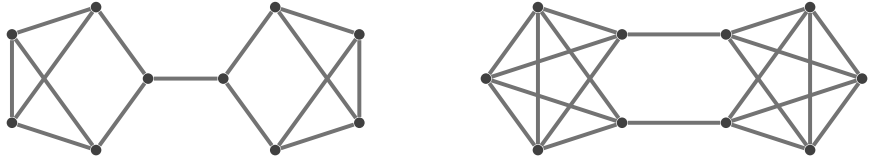
\begin{figure}[!ht]
\centering
    \begin{tikzpicture}
        \node[vertex] (a1) at (0:-1) {};
        \node[vertex] (a2) at (72:-1) {};
        \node[vertex] (a3) at (144:-1) {};
        \node[vertex] (a4) at (-144:-1) {};
        \node[vertex] (a5) at (-72:-1) {};
        \begin{scope}[xshift=-3cm]
            \node[vertex] (b1) at (0:1) {};
            \node[vertex] (b2) at (72:1) {};
            \node[vertex] (b3) at (144:1) {};
            \node[vertex] (b4) at (-144:1) {};
            \node[vertex] (b5) at (-72:1) {};
        \end{scope}
        \draw[edge] (a1)edge(a2) (a1)edge(a5) (a2)edge(a3) (a2)edge(a4) (a3)edge(a4) (a3)edge(a5) (a4)edge(a5);
        \draw[edge] (b1)edge(b2) (b1)edge(b5) (b2)edge(b3) (b2)edge(b4) (b3)edge(b4) (b3)edge(b5) (b4)edge(b5);
        \draw[edge] (a1)edge(b1);
    \end{tikzpicture}\qquad\qquad
    \begin{tikzpicture}
        \node[vertex] (a1) at (0:-1) {};
        \node[vertex] (a2) at (72:-1) {};
        \node[vertex] (a3) at (144:-1) {};
        \node[vertex] (a4) at (-144:-1) {};
        \node[vertex] (a5) at (-72:-1) {};
        \begin{scope}[xshift=3cm]
            \node[vertex] (b1) at (0:1) {};
            \node[vertex] (b2) at (72:1) {};
            \node[vertex] (b3) at (144:1) {};
            \node[vertex] (b4) at (-144:1) {};
            \node[vertex] (b5) at (-72:1) {};
        \end{scope}
        \draw[edge] (a1)edge(a2) (a1)edge(a3) (a1)edge(a4) (a1)edge(a5) (a2)edge(a3) (a2)edge(a4) (a2)edge(a5) (a3)edge(a5) (a4)edge(a5);
        \draw[edge] (b1)edge(b2) (b1)edge(b3) (b1)edge(b4) (b1)edge(b5) (b2)edge(b3) (b2)edge(b4) (b2)edge(b5) (b3)edge(b5) (b4)edge(b5);
        \draw[edge] (a3)edge(b4) (a4)edge(b3);
    \end{tikzpicture}
\caption{(Left) A cubic Ramanujan graph with edge-connectivity one. (Right) A $4$-regular Ramanujan graph with edge-connectivity two}
\label{fig:cubicRamanujan}
\end{figure}

We also notice that there exist 4-regular Ramanujan graphs with $10\le n < 20$ vertices that are not body-hinge globally rigid in $\mathbb{R}^2$. The $4$-regular Ramanujan graph on the right in Figure~\ref{fig:cubicRamanujan} was also constructed in \cite{Cio10}, and its largest absolute eigenvalue smaller than 4 is $\frac{1+\sqrt{33}}{2}<2\sqrt{3}$. Since it has edge-connectivity 2, it is not body-hinge globally rigid in $\mathbb{R}^2$ by Theorem~\ref{thm:bhgr}.

\subsection{Frameworks on surfaces of revolution}

To simplify the problem of dealing with frameworks in three-dimensional space,
we can assume that the joints of our framework are restricted to lie on a smooth surface $\mathcal{M} \subset \mathbb{R}^3$. We assume here that $\mathcal{M}$ is an \textbf{irreducible surface}; i.e.,~$\mathcal{M}$ is the zero set of an irreducible rational polynomial $h(x,y,z) \in \mathbb{Q}[X,Y,Z]$.
The framework $(G,p)$ with $p(v) \in \mathcal{M}$ for every $v \in V(G)$ is \textbf{rigid on $\mathcal{M}$} if there exists  $\varepsilon >0$ such that if $(G,p)$ is equivalent to $(G,q)$ and $\|p(v)-q(v)\|<\epsilon$ and $q(v) \in \mathcal{M}$ for every $v \in V(G)$, then $(G,p)$ is congruent to $(G,q)$.
It was shown in \cite{NixonOwenPower12} that the set of rigid frameworks on an irreducible surface $\mathcal{M}$ either contains an open dense set (in which case we say the graph is \textbf{rigid on $\mathcal{M}$}),
or it is a nowhere dense set.

An irreducible surface is called an \textbf{irreducible surface of revolution} if it can be generated by rotating a continuous curve about a fixed axis. In this special case, we can say the following.

\begin{theorem}[Nixon, Owen and Power \cite{NixonOwenPower12,NixonOwenPower14}]
\label{t:revol}
Let $\mathcal{M}$ be an irreducible surface of revolution.
Then a graph $G$ is rigid on $\mathcal{M}$ if and only if either:
\begin{enumerate}[topsep=1pt,itemsep=0pt,leftmargin=9mm]
    \item $G$ is a complete graph,
    \item $\mathcal{M}$ is a sphere and $G$ contains a spanning Laman graph,
    \item $\mathcal{M}$ is a cylinder and $G$ contains two edge-disjoint spanning trees, or
    \item $\mathcal{M}$ is not a cylinder or a sphere and $G$ contains two edge-disjoint spanning subgraphs $G_1,G_2$, where $G_1$ is a tree and every connected component of $G_2$ contains exactly one cycle.
\end{enumerate}
\end{theorem}

For a surface $\mathcal{M}$ in $\mathbb{R}^3$, we define the following. A framework $(G,p)$ with $p(v) \in \mathcal{M}$ for every $v \in V(G)$ is \textbf{globally rigid on $\mathcal{M}$} if every framework $(G,q)$ that is equivalent to $(G,p)$ with $q(v) \in \mathcal{M}$ for every $v \in V(G)$ is also congruent to $(G,p)$.
For a specific case of $\mathcal{M}$ being the cylinder, it was proven that the set of globally rigid frameworks on the cylinder either contains an open dense set (in which case we say the graph is \textbf{globally rigid on the cylinder}), or it is a nowhere dense set (in which case we say the graph is \textbf{not globally rigid on the cylinder}). They also characterized exactly which graphs are globally rigid on the cylinder.

\begin{theorem}[Jackson and Nixon \cite{JacksonNixon19}]\label{thm:grcyl}
A graph $G=(V,E)$ is globally rigid on the cylinder if and only if either $G$ is a complete graph, or $G$ is 2-connected and $G$ is \textbf{redundantly rigid on the cylinder} (i.e.,~for every edge $e \in E$, the graph $G-e$ is rigid on the cylinder).
\end{theorem}

We now obtain the following sufficient spectral condition for rigidity on certain types of irreducible surfaces of revolution by utilizing the results of Section \ref{sec: connandst}.

\begin{theorem}\label{thm:specgrcyl}
Let $G$ be a graph with minimum degree $\delta$.
\begin{enumerate}[topsep=1pt,itemsep=0pt,leftmargin=9mm]
    \item\label{it:specgrcyl:d4} If $\delta \geq 4$ and $\mu_2(G) > \frac{3}{\delta +1}$, then $G$ is rigid on any irreducible surface of revolution that is not a sphere.
    \item\label{it:specgrcyl:d5} If $\delta \geq 5$ and $\mu_2(G) > \frac{4}{\delta +1}$, then $G$ is redundantly rigid on the cylinder.
\end{enumerate}
\end{theorem}

\begin{proof}
\ref{it:specgrcyl:d4} By Theorem~\ref{thm:LHGL}, $G$ contains 2 edge-disjoint spanning trees. Since $|E| \geq \delta|V|/2 > 2|V| -2$, $G$ must also contain an extra edge which is not in either of the edge-disjoint spanning trees. The result now follows from Theorem \ref{t:revol}.

\ref{it:specgrcyl:d5}
The result follows immediately from Lemma \ref{lem:tgst} and Theorem \ref{t:revol}.
\end{proof}

We now have an immediate corollary for Ramanujan graphs.

\begin{corollary}
Let $G$ be a $k$-regular Ramanujan graph with $n$ vertices where either $k\geq 5$, or $k= 4$ and $n\ge 20$ (or $n\le 9$). Then the following holds.
\begin{enumerate}[topsep=1pt,itemsep=0pt,leftmargin=9mm]
    \item\label{it:k20:rev} $G$ is rigid on any irreducible surface of revolution that is not a sphere.
    \item\label{it:k20:cyl} $G$ is globally rigid on the cylinder.
\end{enumerate}
\end{corollary}

\begin{proof}
\ref{it:k20:rev}
By Proposition~\ref{pro:edgeconn}, $G$ is 4-edge-connected. Thus, $G-e$ contains 2 edge-disjoint spanning trees for every edge $e$. By Theorem \ref{t:revol}, $G$ is rigid on any irreducible surface of revolution that is not a sphere.

\ref{it:k20:cyl}
When $k\geq 5$, we have $\mu_2(G) = k-\lambda_2(G) \ge k -2\sqrt{k-1}> \frac{4}{k+1}$ and thus $G$ is redundantly rigid on the cylinder by Theorem~\ref{thm:specgrcyl}. When $k=4$ and $n\ge 20$ (or $n\le 9$), by Proposition~\ref{pro:edgeconn}, $G$ is 4-edge-connected, and thus $G-e$ contains 2 edge-disjoint spanning trees for every edge $e$ by Theorem \ref{NaTu}. By Theorem~\ref{t:revol}, $G$ is redundantly rigid on the cylinder. By Proposition~\ref{pro:edgeconn}, any $k$-regular Ramanujan graph for $k\ge 4$ is 2-connected. Thus $G$ is globally rigid on the cylinder by Theorem~\ref{thm:grcyl}.
\end{proof}
Notice that there exist 4-regular Ramanujan graphs that are not redundantly rigid (and thus not globally rigid) on the cylinder. In Figure~\ref{fig:cubicRamanujan}, the graph on the right has edge-connectivity 2. Thus, for some edge $e$, $G-e$ has edge-connectivity 1 and so does not have 2 edge-disjoint spanning trees. By Theorem~\ref{t:revol}, it is not redundantly rigid on the cylinder.

\section{Computational results}\label{sec:comp}

The number of regular graphs increases drastically with the number of vertices, and so do the number of Ramanujan graphs. Nevertheless, we were able to obtain some computational results for Ramanujan graphs being rigid or globally rigid in $\mathbb{R}^2$. In this section we summarize those results split up into three subsections, where \Cref{sec:comp:general} deals with the general case, \Cref{sec:comp:bipartite} considers bipartite graphs only and \Cref{sec:comp:vt} specifies for vertex-transitive graphs. In the latter two special cases we can compute until a much higher number of vertices. For the computations we used \geng, which is a part of \nauty, for generating sets of graphs, and our own \mathematica\ code for checking rigidity properties.

\subsection{Ramanujan Graphs}\label{sec:comp:general}

For small orders, it is possible to determine all $k$-regular Ramanujan graphs by computer (see \Cref{tab:ramanujan} and \cite{ZenodoRamanujan} for a data set).
\begin{table}[ht]
    \centering
    \begin{tabular}{rrrrr}
        \toprule
        $|V|$\textbackslash$k$  & 4 & 5 & 6 & 7 \\\midrule
         7 & 2 & - & - & - \\
         8 & 6 & 3 & - & - \\
         9 & 15 & - & 4 & - \\
         10 & 57 & 59 & 21 & 5 \\
         11 & 247 & - & 263 & - \\
         12 & 1476 & 7756 & 7818 & 1544 \\
         13 & 10439 & - & 367121 & - \\
         14 & 85386 & 3429389 & 21566449 & 21603716 \\
         15 & 781675 & - & ? & - \\
         16 & 7777226 & ? & ? & ? \\\bottomrule
    \end{tabular}
    \caption{Number of $k$-regular Ramanujan graphs with given number of vertices.}
    \label{tab:ramanujan}
\end{table}
Using simple implementations of the rigidity properties we were able to check all Ramanujan graphs in the table. For $k \geq 5$, we found that all $k$-regular Ramanujan graphs with at most 14 vertices were global rigid in $\mathbb{R}^2$.
We do know, however, that there exist 5-regular Ramanujan graphs that are not globally rigid with more than 14 vertices;
see \Cref{fig:VertexTransitiveSpecialCase}.

As 4-regular graphs are easier to generate than 5-, 6- and 7-regular graphs,
it was possible for us to compute up to 16 vertex graphs. We found that there are exactly four 4-regular Ramanujan graphs that are not rigid in $\mathbb{R}^2$ with 16 vertices or less (see \Cref{fig:nonrig-ramanujan}).
\begin{figure}[ht]
    \centering
    \begin{tikzpicture}[scale=0.7]
        \node[vertex] (a1) at (0:-1) {};
        \node[vertex] (a2) at (72:-1) {};
        \node[vertex] (a3) at (144:-1) {};
        \node[vertex] (a4) at (-144:-1) {};
        \node[vertex] (a5) at (-72:-1) {};
        \begin{scope}[xshift=3cm]
            \node[vertex] (b1) at (0:1) {};
            \node[vertex] (b2) at (72:1) {};
            \node[vertex] (b3) at (144:1) {};
            \node[vertex] (b4) at (-144:1) {};
            \node[vertex] (b5) at (-72:1) {};
        \end{scope}
        \draw[edge] (a1)edge(a2) (a1)edge(a3) (a1)edge(a4) (a1)edge(a5) (a2)edge(a3) (a2)edge(a4) (a2)edge(a5) (a3)edge(a5) (a4)edge(a5);
        \draw[edge] (b1)edge(b2) (b1)edge(b3) (b1)edge(b4) (b1)edge(b5) (b2)edge(b3) (b2)edge(b4) (b2)edge(b5) (b3)edge(b5) (b4)edge(b5);
        \draw[edge] (a3)edge(b4) (a4)edge(b3);
    \end{tikzpicture}
    \qquad
    \begin{tikzpicture}[scale=0.7]
        \node[vertex] (a1) at (0:-1) {};
        \node[vertex] (a2) at (72:-1) {};
        \node[vertex] (a3) at (144:-1) {};
        \node[vertex] (a4) at (-144:-1) {};
        \node[vertex] (a5) at (-72:-1) {};
        \begin{scope}[xshift=3cm]
            \node[vertex] (b1) at (-45:1) {};
            \node[vertex] (b2) at (45:1) {};
            \node[vertex] (b3) at (135:1) {};
            \node[vertex] (b4) at (-135:1) {};
            \node[vertex] (b5) at (90:0.4) {};
            \node[vertex] (b6) at (-90:0.4) {};
        \end{scope}
        \draw[edge] (a1)edge(a2) (a1)edge(a3) (a1)edge(a4) (a1)edge(a5) (a2)edge(a3) (a2)edge(a4) (a2)edge(a5) (a3)edge(a5) (a4)edge(a5);
        \draw[edge] (b1)edge(b2) (b1)edge(b4) (b1)edge(b5) (b1)edge(b6) (b2)edge(b3) (b2)edge(b5) (b2)edge(b6) (b3)edge(b5) (b3)edge(b6) (b4)edge(b5) (b4)edge(b6);
        \draw[edge] (a3)edge(b4) (a4)edge(b3);
    \end{tikzpicture}
    
    \begin{tikzpicture}[scale=0.5]
        \coordinate (o) at (0,0);
        \foreach \x in {1,2,3,4}
        {
            \begin{scope}[rotate around={\x*90:(o)}]
                \begin{scope}[xshift=2cm]
                    \node[vertex] (\x1) at (0:1) {};
                    \node[vertex] (\x2) at (90:1) {};
                    \node[vertex] (\x3) at (180:1) {};
                    \node[vertex] (\x4) at (270:1) {};
                \end{scope}
                \draw[edge] (\x1)edge(\x2) (\x1)edge(\x3) (\x1)edge(\x4) (\x2)edge(\x3) (\x2)edge(\x4) (\x3)edge(\x4);
            \end{scope}
        }
        \draw[edge] (12)edge(24) (22)edge(34) (32)edge(44) (42)edge(14);
        \draw[edge] (13)edge(33) (23)edge(43);
    \end{tikzpicture}
    \qquad
    \begin{tikzpicture}[scale=0.5]
        \coordinate (o) at (0,0);
        \foreach \x in {1,2,3,4}
        {
            \begin{scope}[rotate around={\x*90:(o)}]
                \begin{scope}[xshift=2cm]
                    \node[vertex] (\x1) at (-45:1) {};
                    \node[vertex] (\x2) at (45:1) {};
                    \node[vertex] (\x3) at (135:1) {};
                    \node[vertex] (\x4) at (-135:1) {};
                \end{scope}
                \draw[edge] (\x1)edge(\x2) (\x1)edge(\x3) (\x1)edge(\x4) (\x2)edge(\x3) (\x2)edge(\x4) (\x3)edge(\x4);
            \end{scope}
        }
        \draw[edge] (12)edge(21) (22)edge(31) (32)edge(41) (42)edge(11);
        \draw[edge] (13)edge(24) (23)edge(34) (33)edge(44) (43)edge(14);
    \end{tikzpicture}
    \caption{The only $4$-regular Ramanujan graphs with at most 16 vertices that are not rigid in $\mathbb{R}^2$.}
    \label{fig:nonrig-ramanujan}
\end{figure}
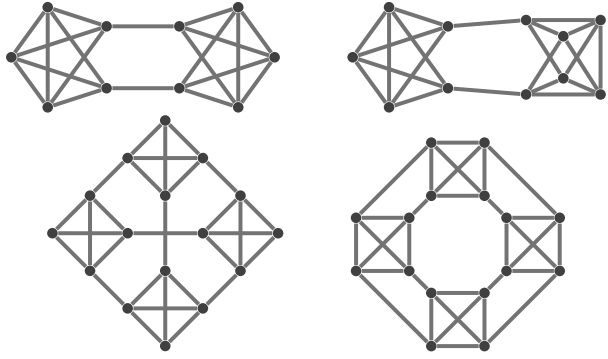
There are, however, plenty of rigid $4$-regular Ramanujan graphs with at most 16 vertices that are not globally rigid in $\mathbb{R}^2$ as shown in \Cref{tab:ngrig-ramanujan} (see also \cite{Zenodo4Regular} for a data set).
\begin{table}[ht]
    \centering
    \begin{tabular}{*{8}{r}}
        \toprule
         $|V|$  & 10 & 11 & 12 & 13 & 14  & 15   & 16\\
         graphs & 1  & 3  & 17 & 70 & 340 & 1573 & 7425\\\bottomrule
    \end{tabular}
    \caption{The number of $4$-regular Ramanujan graphs with $n \leq 16$ vertices that are rigid but not globally rigid in $\mathbb{R}^2$.}
    \label{tab:ngrig-ramanujan}
\end{table}

\subsection{Bipartite Ramanujan Graphs}\label{sec:comp:bipartite}
For bipartite Ramanujan graphs we can go slightly further, as \geng\ allows us to compute regular bipartite graphs directly. Using this, we were able to compute all the $k$-regular Ramanujan graphs for $k\in \{4,5,6,7\}$ up to 20 vertices (see \Cref{tab:bramanujan} and \cite{ZenodoRamanujan} for a data set); for $4$-regularity we even managed to compute until 22 vertices. 
\begin{table}[ht]
    \centering
    \begin{tabular}{rrrrr}
        \toprule
        $|V|$\textbackslash$k$  & 4 & 5 & 6 & 7 \\\midrule
        8 & 1 & - & - & - \\
        10 & 1 & 1 & - & - \\
        12 & 4 & 1 & 1 & - \\
        14 & 14 & 4 & 1 & 1 \\
        16 & 128 & 41 & 7 & 1 \\
        18 & 1973 & 1981 & 157 & 8 \\
        20 & 62447 & 304470 & 62616 & 725 \\
        22 & 2801916 & ? & ? & ? \\\bottomrule
    \end{tabular}
    \caption{Number of $k$-regular bipartite Ramanujan graphs with given number of vertices.}
    \label{tab:bramanujan}
\end{table}
We computed that if $k \geq 5$, every $k$-regular bipartite Ramanujan graph with at most 20 vertices is globally rigid in $\mathbb{R}^2$. We can combine this computational result with \Cref{thm:ramanujan} to obtain the following result.

\begin{corollary}
    Every 7-regular bipartite Ramanujan graph is globally rigid in $\mathbb{R}^2$.
\end{corollary}

For 4-regular bipartite Ramanujan graphs, the situation is slightly more complicated. We discovered that every 4-regular bipartite Ramanujan graph with up to 22 vertices is rigid in $\mathbb{R}^2$, and every 4-regular bipartite Ramanujan graph with up to 22 vertices that is not one of the two graphs pictured in \Cref{fig:ngrig-bramanujan} is globally rigid in $\mathbb{R}^2$.
\begin{figure}[ht]
    \centering
    \begin{tikzpicture}
        \node[vertex] (a1) at (0,0.4) {};
        \node[vertex] (a2) at (0,-0.4) {};
        \node[vertex] (a3) at (0.5,1.1) {};
        \node[vertex] (a4) at (0.5,-1.1) {};
        \node[vertex] (a5) at (0.8,0) {};
        \node[vertex] (a6) at (1.2,0.6) {};
        \node[vertex] (a7) at (1.2,-0.6) {};
        \draw[edge] (a1)edge(a3) (a1)edge(a4) (a1)edge(a6) (a1)edge(a7) (a2)edge(a3) (a2)edge(a4) (a2)edge(a6) (a2)edge(a7) (a5)edge(a3) (a5)edge(a4) (a5)edge(a6) (a5)edge(a7);
        \node[vertex] (c1) at (2,1) {};
        \node[vertex] (c2) at (2,-1) {};
        \node[vertex] (b1) at ($2*(c1)!(a1)!(c2)-(a1)$) {};
        \node[vertex] (b2) at ($2*(c1)!(a2)!(c2)-(a2)$) {};
        \node[vertex] (b3) at ($2*(c1)!(a3)!(c2)-(a3)$) {};
        \node[vertex] (b4) at ($2*(c1)!(a4)!(c2)-(a4)$) {};
        \node[vertex] (b5) at ($2*(c1)!(a5)!(c2)-(a5)$) {};
        \node[vertex] (b6) at ($2*(c1)!(a6)!(c2)-(a6)$) {};
        \node[vertex] (b7) at ($2*(c1)!(a7)!(c2)-(a7)$) {};
        \draw[edge] (b1)edge(b3) (b1)edge(b4) (b1)edge(b6) (b1)edge(b7) (b2)edge(b3) (b2)edge(b4) (b2)edge(b6) (b2)edge(b7) (b5)edge(b3) (b5)edge(b4) (b5)edge(b6) (b5)edge(b7);
        \draw[edge] (a3)edge(c1) (a6)edge(c1) (a4)edge(c2) (a7)edge(c2);
        \draw[edge] (b3)edge(c1) (b6)edge(c1) (b4)edge(c2) (b7)edge(c2);
    \end{tikzpicture}
    \qquad
    \begin{tikzpicture}[rotate=-90,scale=0.8]
        \begin{scope}[rotate around={20:(2,0)}]
            \node[vertex] (a1) at (0,0.4) {};
            \node[vertex] (a2) at (0,-0.4) {};
            \node[vertex] (a3) at (0.5,1.1) {};
            \node[vertex] (a4) at (0.5,-1.1) {};
            \node[vertex] (a5) at (0.8,0) {};
            \node[vertex] (a6) at (1.2,0.6) {};
            \node[vertex] (a7) at (1.2,-0.6) {};
        \end{scope}
        \draw[edge] (a1)edge(a3) (a1)edge(a4) (a1)edge(a6) (a1)edge(a7) (a2)edge(a3) (a2)edge(a4) (a2)edge(a6) (a2)edge(a7) (a5)edge(a3) (a5)edge(a4) (a5)edge(a6) (a5)edge(a7);
        \coordinate (c1) at (2,1);
        \coordinate (c2) at (2,-1);
        \node[vertex] (b1) at ($2*(c1)!(a1)!(c2)-(a1)$) {};
        \node[vertex] (b2) at ($2*(c1)!(a2)!(c2)-(a2)$) {};
        \node[vertex] (b3) at ($2*(c1)!(a3)!(c2)-(a3)$) {};
        \node[vertex] (b4) at ($2*(c1)!(a4)!(c2)-(a4)$) {};
        \node[vertex] (b5) at ($2*(c1)!(a5)!(c2)-(a5)$) {};
        \node[vertex] (b6) at ($2*(c1)!(a6)!(c2)-(a6)$) {};
        \node[vertex] (b7) at ($2*(c1)!(a7)!(c2)-(a7)$) {};
        \begin{scope}[xshift=2cm,yshift=1cm]
            \node[vertex] (d1) at (-0.75,0) {};
            \node[vertex] (d2) at (0.75,0) {};
            \node[vertex] (d3) at (-1,0.75) {};
            \node[vertex] (d4) at (1,0.75) {};
            \node[vertex] (d5) at (-1,1.5) {};
            \node[vertex] (d6) at (1,1.5) {};
            \node[vertex] (d7) at (-1,2.25) {};
            \node[vertex] (d8) at (1,2.25) {};
        \end{scope}
        \draw[edge] (d1)edge(d4) (d3)edge(d2) (d1)edge(d6) (d1)edge(d8) (d3)edge(d6) (d3)edge(d8)  (d5)edge(d2) (d5)edge(d4) (d5)edge(d6) (d5)edge(d8) (d7)edge(d2) (d7)edge(d4) (d7)edge(d6) (d7)edge(d8);
        
        \draw[edge] (b1)edge(b3) (b1)edge(b4) (b1)edge(b6) (b1)edge(b7) (b2)edge(b3) (b2)edge(b4) (b2)edge(b6) (b2)edge(b7) (b5)edge(b3) (b5)edge(b4) (b5)edge(b6) (b5)edge(b7);
        \draw[edge] (a4)edge(b4) (a7)edge(b7);
        \draw[edge] (a3)edge(d3) (a6)edge(d1) (b3)edge(d4) (b6)edge(d2);
    \end{tikzpicture}
    \caption{$4$-regular bipartite Ramanujan graphs that are rigid but not globally rigid in $\mathbb{R}^2$.}
    \label{fig:ngrig-bramanujan}
\end{figure}
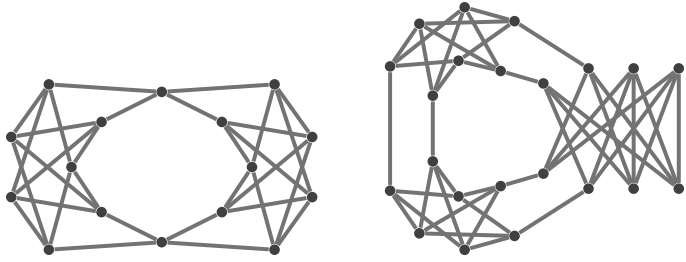

However, there do exist $4$-regular bipartite Ramanujan graphs that are not rigid in $\mathbb{R}^2$ with more than 22 vertices. A similar construction as in \Cref{fig:ngrig-bramanujan} yielded one such bipartite graph (see \Cref{fig:nonrig-bramanujan}).
\begin{figure}[ht]
    \centering
    \begin{tikzpicture}[scale=0.8]
        \coordinate (o) at (0,0);
        \foreach \x in {1,2,3,4}
        {
            \begin{scope}[rotate around=\x*90-45:(o)]
                \begin{scope}[xshift=3.5cm,xscale=-1]
                    \node[vertex] (\x1) at (0,0.4) {};
                    \node[vertex] (\x2) at (0,-0.4) {};
                    \node[vertex] (\x3) at (0.5,1.1) {};
                    \node[vertex] (\x4) at (0.5,-1.1) {};
                    \node[vertex] (\x5) at (0.8,0) {};
                    \node[vertex] (\x6) at (1.2,0.6) {};
                    \node[vertex] (\x7) at (1.2,-0.6) {};
                \end{scope}
            \end{scope}
            \draw[edge] (\x1)edge(\x3) (\x1)edge(\x4) (\x1)edge(\x6) (\x1)edge(\x7) (\x2)edge(\x3) (\x2)edge(\x4) (\x2)edge(\x6) (\x2)edge(\x7) (\x5)edge(\x3) (\x5)edge(\x4) (\x5)edge(\x6) (\x5)edge(\x7);
        }
        \node[vertex] (c1) at (0:1.9) {};
        \node[vertex] (c2) at (90:1.9) {};
        \node[vertex] (c3) at (180:1.9) {};
        \node[vertex] (c4) at (270:1.9) {};
        \draw[edge] (13)edge(c2) (16)edge(c3) (14)edge(c1) (17)edge(c4);
        \draw[edge] (23)edge(c3) (26)edge(c4) (24)edge(c2) (27)edge(c1);
        \draw[edge] (33)edge(c4) (36)edge(c1) (34)edge(c3) (37)edge(c2);
        \draw[edge] (43)edge(c1) (46)edge(c2) (44)edge(c4) (47)edge(c3);
    \end{tikzpicture}
    \caption{A $4$-regular bipartite Ramanujan graph that is not rigid in $\mathbb{R}^2$.}
    \label{fig:nonrig-bramanujan}
\end{figure}
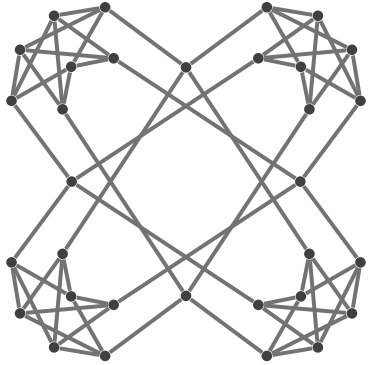

\subsection{Vertex-Transitive Ramanujan Graphs}\label{sec:comp:vt}
Since there are comparably few vertex-transitive regular graphs, we were able to use the precompiled lists from \cite{RH-Data} to obtain all the vertex-transitive Ramanujan graphs up to 47 vertices (see \cite{RH} for more details on how the list was compiled). Because of \Cref{thm:transitive}, we are only interested in the 4-regular vertex-transitive  Ramanujan graphs. \Cref{tab:vramanujan} shows how many 4-regular vertex-transitive Ramanujan graphs there exist with $n$ vertices for $7 \leq n \leq 47$ (see \cite{ZenodoRamanujan} for a data set).
\begin{table}[ht]
    \centering
    \begin{tabular}{*{50}{r}}
        \toprule
        $|V|$ & 10 & 11 & 12 & 13 & 14 & 15 & 16 & 17 & 18 & 19 & 20 & 21 & 22\\
        graphs & 4 & 2 & 11 & 3 & 6 & 8 & 16 & 4 & 16 & 4 & 28 & 11 & 11 \\\bottomrule
        \toprule
        $|V|$ & 23 & 24 & 25 & 26 & 27 & 28 & 29 & 30 & 31 & 32 & 33 & 34 & 35\\
        graphs & 5 & 74 & 9 & 16 & 16 & 34 & 7 & 52 & 7 & 80 & 14 & 23 & 15\\\bottomrule
        \toprule
        $|V|$  & 36 & 37 & 38 & 39 & 40 & 41 & 42 & 43 & 44 & 45 & 46 & 47 \\
        graphs & 116 & 9 & 27 & 19 & 133 & 10 & 81 & 10 & 65 & 33 & 36 & 11 \\\bottomrule
    \end{tabular}
    \caption{Number of $4$-regular vertex-transitive Ramanujan graphs with given number of vertices.}
    \label{tab:vramanujan}
\end{table}
Out of these graphs, there are six that are not rigid in $\mathbb{R}^2$ (see \Cref{fig:nonrig-vramanujan}), and only one that is rigid but not globally rigid in $\mathbb{R}^2$ (the graph on the left in \Cref{fig:secondSpecialCase}). We can also deduce from Theorem~\ref{t:jjs07} that any 4-regular vertex-transitive Ramanujan graph with 49, 50 or 51 vertices must be globally rigid in $\mathbb{R}^2$, since the number of vertices of any 4-regular vertex-transitive graph with every vertex in exactly one 4-clique must be divisible by 4. The only cases left to check are those with 48 or 52 vertices. It is possible that there are no non-rigid 4-regular vertex-transitive Ramanujan graphs with 52, or even 48, vertices, in which case the bound given by Theorem~\ref{thm:transitive} could potentially be reduced from 53.
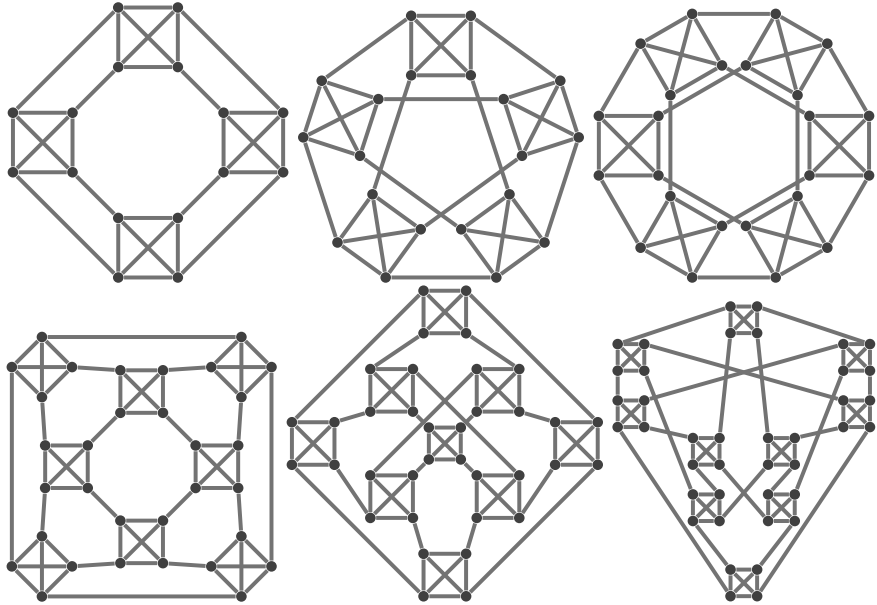
\begin{figure}[ht]
    \centering
    \begin{tikzpicture}[scale=0.7]
        \coordinate (o) at (0,0);
        \foreach \x in {1,2,3,4}
        {
            \begin{scope}[rotate around={\x*90:(o)}]
                \begin{scope}[xshift=2cm]
                    \node[vertex] (\x1) at (-45:0.8) {};
                    \node[vertex] (\x2) at (45:0.8) {};
                    \node[vertex] (\x3) at (135:0.8) {};
                    \node[vertex] (\x4) at (-135:0.8) {};
                \end{scope}
                \draw[edge] (\x1)edge(\x2) (\x1)edge(\x3) (\x1)edge(\x4) (\x2)edge(\x3) (\x2)edge(\x4) (\x3)edge(\x4);
            \end{scope}
        }
        \draw[edge] (12)edge(21) (22)edge(31) (32)edge(41) (42)edge(11);
        \draw[edge] (13)edge(24) (23)edge(34) (33)edge(44) (43)edge(14);
    \end{tikzpicture}
    \begin{tikzpicture}[scale=0.7,rotate=90]
        \coordinate (o) at (0,0);
        \foreach \x in {1,2,3,4,5}
        {
            \begin{scope}[rotate around={\x*72:(o)}]
                \begin{scope}[xshift=2cm]
                    \node[vertex] (\x1) at (-45:0.8) {};
                    \node[vertex] (\x2) at (45:0.8) {};
                    \node[vertex] (\x3) at (135:0.8) {};
                    \node[vertex] (\x4) at (-135:0.8) {};
                \end{scope}
                \draw[edge] (\x1)edge(\x2) (\x1)edge(\x3) (\x1)edge(\x4) (\x2)edge(\x3) (\x2)edge(\x4) (\x3)edge(\x4);
            \end{scope}
        }
        \draw[edge] (12)edge(21) (22)edge(31) (32)edge(41) (42)edge(51) (52)edge(11);
        \draw[edge] (13)edge(34) (23)edge(44) (33)edge(54) (43)edge(14) (53)edge(24);
    \end{tikzpicture}
    \begin{tikzpicture}[scale=0.7]
        \coordinate (o) at (0,0);
        \foreach \x in {1,2,3,4,5,6}
        {
            \begin{scope}[rotate around={\x*60:(o)}]
                \begin{scope}[xshift=2cm]
                    \node[vertex] (\x1) at (-45:0.8) {};
                    \node[vertex] (\x2) at (45:0.8) {};
                    \node[vertex] (\x3) at (135:0.8) {};
                    \node[vertex] (\x4) at (-135:0.8) {};
                \end{scope}
                \draw[edge] (\x1)edge(\x2) (\x1)edge(\x3) (\x1)edge(\x4) (\x2)edge(\x3) (\x2)edge(\x4) (\x3)edge(\x4);
            \end{scope}
        }
        \draw[edge] (12)edge(21) (22)edge(31) (32)edge(41) (42)edge(51) (52)edge(61) (62)edge(11);
        \draw[edge] (13)edge(34) (23)edge(44) (33)edge(54) (43)edge(64) (53)edge(14) (63)edge(24);
    \end{tikzpicture}
    
    \begin{tikzpicture}[scale=0.5]
        \coordinate (o) at (0,0);
        \foreach \x in {1,2,3,4}
        {
            \begin{scope}[rotate around={\x*90:(o)}]
                \begin{scope}[xshift=2cm]
                    \node[vertex] (\x1) at (-45:0.8) {};
                    \node[vertex] (\x2) at (45:0.8) {};
                    \node[vertex] (\x3) at (135:0.8) {};
                    \node[vertex] (\x4) at (-135:0.8) {};
                \end{scope}
                \draw[edge] (\x1)edge(\x2) (\x1)edge(\x3) (\x1)edge(\x4) (\x2)edge(\x3) (\x2)edge(\x4) (\x3)edge(\x4);
            \end{scope}
        }
        \foreach \x in {5,6,7,8}
        {
            \begin{scope}[rotate around={\x*90-45:(o)}]
                \begin{scope}[xshift=3.75cm]
                    \node[vertex] (\x1) at (-45:0.8) {};
                    \node[vertex] (\x2) at (45:0.8) {};
                    \node[vertex] (\x3) at (135:0.8) {};
                    \node[vertex] (\x4) at (-135:0.8) {};
                \end{scope}
                \draw[edge] (\x1)edge(\x2) (\x1)edge(\x3) (\x1)edge(\x4) (\x2)edge(\x3) (\x2)edge(\x4) (\x3)edge(\x4);
            \end{scope}
        }
        \draw[edge] (52)edge(61) (62)edge(71) (72)edge(81) (82)edge(51);
        \draw[edge] (13)edge(24) (23)edge(34) (33)edge(44) (43)edge(14);
        \draw[edge] (11)edge(53) (21)edge(63) (31)edge(73) (41)edge(83);
        \draw[edge] (12)edge(64) (22)edge(74) (32)edge(84) (42)edge(54);
    \end{tikzpicture}
    \begin{tikzpicture}[rotate=-135,scale=0.5]
        \coordinate (o) at (0,0);
        \node[vertex] (01) at (0:0.6) {};
        \node[vertex] (02) at (90:0.6) {};
        \node[vertex] (03) at (180:0.6) {};
        \node[vertex] (04) at (270:0.6) {};
        \draw[edge] (01)edge(02) (01)edge(03) (01)edge(04) (02)edge(03) (02)edge(04) (03)edge(04);
        \foreach \x in {1,2,3,4}
        {
            \begin{scope}[rotate around={\x*90:(o)}]
                \begin{scope}[xshift=2cm]
                    \node[vertex] (\x1) at (0:0.8) {};
                    \node[vertex] (\x2) at (90:0.8) {};
                    \node[vertex] (\x3) at (180:0.8) {};
                    \node[vertex] (\x4) at (270:0.8) {};
                \end{scope}
                \draw[edge] (\x1)edge(\x2) (\x1)edge(\x3) (\x1)edge(\x4) (\x2)edge(\x3) (\x2)edge(\x4) (\x3)edge(\x4);
            \end{scope}
        }
        \draw[edge] (13)edge(02) (33)edge(04) (23)edge(03) (43)edge(01);
        \foreach \x in {5,6,7,8}
        {
            \begin{scope}[rotate around={\x*90-45:(o)}]
                \begin{scope}[xshift=3.5cm]
                    \node[vertex] (\x1) at (-45:0.8) {};
                    \node[vertex] (\x2) at (45:0.8) {};
                    \node[vertex] (\x3) at (135:0.8) {};
                    \node[vertex] (\x4) at (-135:0.8) {};
                \end{scope}
                \draw[edge] (\x1)edge(\x2) (\x1)edge(\x3) (\x1)edge(\x4) (\x2)edge(\x3) (\x2)edge(\x4) (\x3)edge(\x4);
            \end{scope}
        }
        \draw[edge] (52)edge(61) (62)edge(71) (72)edge(81) (82)edge(51);
        \draw[edge] (12)edge(34) (22)edge(44);
        \draw[edge] (21)edge(74) (31)edge(73);
        \draw[edge] (14)edge(53) (42)edge(54);
        \draw[edge] (11)edge(64) (24)edge(63);
        \draw[edge] (41)edge(83) (32)edge(84);
    \end{tikzpicture}
    \begin{tikzpicture}[scale=0.5]
        \coordinate (o) at (0,0);
        \foreach \x [count=\i] in {(0,-0.5),(-1,1.5),(1,1.5),(-1,3),(1,3),(-3,4),(3,4),(-3,5.5),(3,5.5),(0,6.5)}
        {
            \begin{scope}[shift=\x]
                \node[vertex] (\i1) at (-45:0.5) {};
                \node[vertex] (\i2) at (45:0.5) {};
                \node[vertex] (\i3) at (135:0.5) {};
                \node[vertex] (\i4) at (-135:0.5) {};
            \end{scope}
            \draw[edge] (\i1)edge(\i2) (\i1)edge(\i3) (\i1)edge(\i4) (\i2)edge(\i3) (\i2)edge(\i4) (\i3)edge(\i4);
        }
        \draw[edge] (14)edge(64) (13)edge(24) (12)edge(31) (11)edge(71);
        \draw[edge] (21)edge(54) (22)edge(44) (34)edge(41) (33)edge(51);
        \draw[edge] (62)edge(93) (63)edge(84) (73)edge(82) (72)edge(91);
        \draw[edge] (61)edge(43) (81)edge(23) (94)edge(32) (74)edge(52);
        \draw[edge] (101)edge(53) (102)edge(92) (103)edge(83) (104)edge(42);
    \end{tikzpicture}
    \caption{Vertex-transitive Ramanujan graphs that are not rigid in $\mathbb{R}^2$.}
    \label{fig:nonrig-vramanujan}
\end{figure}

\section{Open problems}\label{sec:end}

It is currently unknown for exactly which values of $k$ the following three statements hold.

\textbf{(i) There exist only finitely many $k$-regular Ramanujan graphs that are not (globally) rigid in $\mathbb{R}^2$.} With Theorem \ref{thm:ramanujan}, we know that the statement is true for $k \geq 6$. We can also easily see that it is false for $k \leq 3$, as there are only three rigid cubic graphs (the complete graph $K_4$, the complete bipartite graph $K_{3,3}$ and the Cartesian product $K_2 \Box K_3$). However, the cases of $k \in \{4,5\}$ still remain open.

\textbf{(ii) All $k$-regular Ramanujan graphs are (globally) rigid in $\mathbb{R}^2$.}
As stated previously, we know that the statement is false for $k \leq 3$. The statement is also false for $k \in \{4,5\}$, as shown by Figures \ref{fig:VertexTransitiveSpecialCase} and \ref{fig:nonrig-ramanujan}. The statement is true for $k \geq 8$ by \Cref{specrigid}, so the remaining open cases are $k \in \{6,7\}$. We do know from our computational results (see \Cref{sec:comp}) that any 6- or 7-regular Ramanujan graph that is not (globally) rigid must have more than 14 vertices. This implies that the case of $k=7$ can be verified by computing all the 7-regular Ramanujan graphs with 16, 18 or 20 vertices. This is, however, currently beyond the computing power we have access to.
    
\textbf{(iii) There exist only finitely many $k$-regular Ramanujan graphs that are not 3-edge-connected.}
By \Cref{thm:bhgr}. this is equivalent to checking whether there are only finitely many $k$-regular Ramanujan graphs that are not body-hinge globally rigid in $\mathbb{R}^2$. The statement is obviously false for $k \leq 2$, and is true for $k \geq 4$ by \Cref{pro:edgeconn}. Hence, the only open case is $k = 3$. Using the computational methods laid out in \Cref{sec:comp}, we constructed a list of all the cubic Ramanujan graphs with 20 vertices or less. Within this list, we found only 4 cubic Ramanujan graphs with edge-connectivity 1; one with 10 vertices (see the graph on the left in \Cref{fig:cubicRamanujan}) and three with 12 vertices (see \Cref{fig:cubicRamanujanOneconnected}).
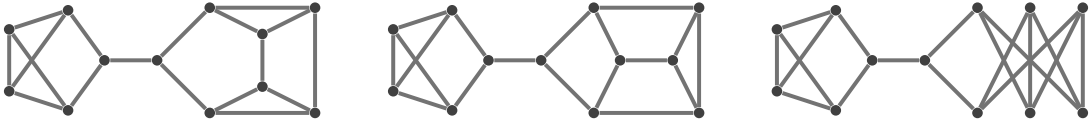
\begin{figure}[ht]
    \centering
    \begin{tikzpicture}[scale=0.7]
        \begin{scope}
            \node[vertex] (a1) at (0:1) {};
            \node[vertex] (a2) at (72:1) {};
            \node[vertex] (a3) at (144:1) {};
            \node[vertex] (a4) at (-144:1) {};
            \node[vertex] (a5) at (-72:1) {};
        \end{scope}
        \begin{scope}[xshift=4cm]
            \node[vertex] (b1) at (-1,-1) {};
            \node[vertex] (b2) at (1,-1) {};
            \node[vertex] (b3) at (0,-0.5) {};
            \node[vertex] (b4) at (-1,1) {};
            \node[vertex] (b5) at (1,1) {};
            \node[vertex] (b6) at (0,0.5) {};
            \node[vertex] (b7) at (-2,0) {};
        \end{scope}
        \draw[edge] (a1)edge(a2) (a1)edge(a5) (a2)edge(a3) (a2)edge(a4) (a3)edge(a4) (a3)edge(a5) (a4)edge(a5);
        \draw[edge] (b1)edge(b2) (b1)edge(b3) (b2)edge(b3) (b2)edge(b5) (b3)edge(b6) (b4)edge(b5) (b4)edge(b6) (b5)edge(b6);
         \draw[edge] (b1)edge(b7) (b4)edge(b7);
        \draw[edge] (a1)edge(b7);
    \end{tikzpicture}
    \qquad
    \begin{tikzpicture}[scale=0.7]
        \begin{scope}
            \node[vertex] (a1) at (0:1) {};
            \node[vertex] (a2) at (72:1) {};
            \node[vertex] (a3) at (144:1) {};
            \node[vertex] (a4) at (-144:1) {};
            \node[vertex] (a5) at (-72:1) {};
        \end{scope}
        \begin{scope}[xshift=4cm]
            \begin{scope}[rotate=-90]
                \node[vertex] (b1) at (-1,-1) {};
                \node[vertex] (b2) at (1,-1) {};
                \node[vertex] (b3) at (0,-0.5) {};
                \node[vertex] (b4) at (-1,1) {};
                \node[vertex] (b5) at (1,1) {};
                \node[vertex] (b6) at (0,0.5) {};
            \end{scope}
            \node[vertex] (b7) at (-2,0) {};
        \end{scope}
        \draw[edge] (a1)edge(a2) (a1)edge(a5) (a2)edge(a3) (a2)edge(a4) (a3)edge(a4) (a3)edge(a5) (a4)edge(a5);
        \draw[edge] (b1)edge(b3) (b2)edge(b3) (b1)edge(b4) (b2)edge(b5) (b3)edge(b6) (b4)edge(b5) (b4)edge(b6) (b5)edge(b6);
         \draw[edge] (b1)edge(b7) (b2)edge(b7);
        \draw[edge] (a1)edge(b7);
    \end{tikzpicture}
    \qquad
    \begin{tikzpicture}[scale=0.7]
        \begin{scope}
            \node[vertex] (a1) at (0:1) {};
            \node[vertex] (a2) at (72:1) {};
            \node[vertex] (a3) at (144:1) {};
            \node[vertex] (a4) at (-144:1) {};
            \node[vertex] (a5) at (-72:1) {};
        \end{scope}
        \begin{scope}[xshift=4cm]
            \node[vertex] (b1) at (-1,-1) {};
            \node[vertex] (b3) at (0,-1) {};
            \node[vertex] (b5) at (1,-1) {};
            \node[vertex] (b2) at (-1,1) {};
            \node[vertex] (b4) at (0,1) {};
            \node[vertex] (b6) at (1,1) {};
            \node[vertex] (b7) at (-2,0) {};
        \end{scope}
        \draw[edge] (a1)edge(a2) (a1)edge(a5) (a2)edge(a3) (a2)edge(a4) (a3)edge(a4) (a3)edge(a5) (a4)edge(a5);
        \draw[edge] (b1)edge(b4) (b1)edge(b6) (b3)edge(b2) (b3)edge(b4) (b3)edge(b6) (b5)edge(b2) (b5)edge(b4) (b5)edge(b6);
         \draw[edge] (b1)edge(b7) (b2)edge(b7);
        \draw[edge] (a1)edge(b7);
    \end{tikzpicture}
\caption{All the cubic Ramanujan graphs with edge-connectivity one and 12 vertices.}
\label{fig:cubicRamanujanOneconnected}
\end{figure}
As can be seen by the following result, these 4 graphs are in fact the only cubic Ramanujan graphs with edge-connectivity 1.

\begin{proposition}
A connected cubic graph with at least $22$ vertices and edge-connectivity 1 cannot be Ramanujan.
\end{proposition}
\begin{proof} 
Let $G=(V,E)$ be a cubic graph with edge-connectivity one and $n\geq 22$ vertices. Denote by $a_1a_2$ a cut-edge of it. Define $G_1$ as the subgraph induced by all those vertices $x$ such that the shortest path between $x$ and $a_2$ goes through $a_1$. The subgraph $G_2$ can be defined similarly. Denote by $n_j$ the number of vertices in $G_j$ for $j=1,2$ and assume that $n_1\geq n_2$. We also have that $n_1+n_2=n$, $n_1\geq 11$, $n_2\geq 5$, and both $n_1$ and $n_2$ are odd.
    
Denote by $H_1$ the subgraph of $G$ obtained from $G_1$ by removing $a_1$. Take the disjoint union $H$ of $H_1$ and $G_2$. This is an induced subgraph of $G$ and by Cauchy eigenvalue interlacing, we have that
\begin{equation}\label{eq:cubicinterlace}
    \lambda_2(G)\geq \lambda_2(H)\geq \lambda_1(H) = \min \{\lambda_1(H_1),\lambda_1(G_2)\}.
\end{equation}
We recall that the largest eigenvalue of a graph is always at least its average degree (with equality if the graph is regular). The subgraph $H_1$ has average degree $\frac{3(n_1-3)+4}{n_1-1}=3-\frac{2}{n_1-1}$. If $n_1\geq 13$, this is at least $3-1/6>2\sqrt{2}$. Therefore, $\lambda_1(H_1)>2\sqrt{2}$ in this case. If $n_1=11$, then the average degree of $H_1$ is $2.8$. Using \cite[Lemma 5]{CG07}, we get that $\lambda_1(H_1)-2.8>\frac{1}{100}\cdot \frac{2(n_1-3)^2}{3n_1-5}>0.45$, leading to $\lambda_1(H_1)>2.845>2\sqrt{2}$ in this case as well. If $n_2=5$, then (since $G_2$ is the graph formed from $K_4$ by subdividing one edge) $\lambda_1(G_2)$ equals the largest root of the cubic equation $x^3-x^2-6x+2=0$ and is larger than $2.85>2\sqrt{2}$ (see \cite[Lemma 6]{CGH09} for example). If $n_2\geq 7$, then the average degree of $G_2$ is $3-\frac{1}{n_2}\geq 3-\frac{1}{7}>2\sqrt{2}$. Hence, by inequality \eqref{eq:cubicinterlace}, the graph $G$ cannot be Ramanujan when $n\geq 22$.
\end{proof}

Unlike with the edge-connectivity 1 case, we found a lot of cubic Ramanujan graphs with edge-connectivity 2 and at most 20 vertices; for example, there are exactly 85046 cubic Ramanujan graphs with 20 vertices and edge-connectivity 2, (see \Cref{fig:cubicRamanujanTwoconnected} for two such graphs). We believe, however, that there exist only finitely many cubic Ramanujan graphs with edge-connectivity~2, which we leave as an open question.

\begin{figure}[ht]
    \centering
    \begin{tikzpicture}[yscale=0.6]
        \begin{scope}[xshift=-4cm]
            \node[vertex] (a1) at (-1,-1) {};
            \node[vertex] (a3) at (0,-1) {};
            \node[vertex] (a5) at (1,1) {};
            \node[vertex] (a2) at (-1,1) {};
            \node[vertex] (a4) at (0,1) {};
            \node[vertex] (a6) at (1,-1) {};
            \node[vertex] (a7) at (-0.5,-2) {};
            \node[vertex] (a8) at (-0.5,2) {};
            \node[vertex] (a9) at (1.5,-2) {};
            \node[vertex] (a10) at (1.5,2) {};
        \end{scope}
        \begin{scope}[xscale=-1]
            \node[vertex] (b1) at (-1,-1) {};
            \node[vertex] (b3) at (0,-1) {};
            \node[vertex] (b5) at (1,1) {};
            \node[vertex] (b2) at (-1,1) {};
            \node[vertex] (b4) at (0,1) {};
            \node[vertex] (b6) at (1,-1) {};
            \node[vertex] (b7) at (-0.5,-2) {};
            \node[vertex] (b8) at (-0.5,2) {};
            \node[vertex] (b9) at (1.5,-2) {};
            \node[vertex] (b10) at (1.5,2) {};
        \end{scope}
        \draw[edge] (a1)edge(a2) (a1)edge(a4) (a3)edge(a2) (a3)edge(a6) (a5)edge(a4) (a5)edge(a6);
        \draw[edge] (a1)edge(a7) (a3)edge(a7) (a7)edge(a9) (a6)edge(a9);
        \draw[edge] (a2)edge(a8) (a4)edge(a8) (a8)edge(a10) (a5)edge(a10);
        \draw[edge] (b1)edge(b2) (b1)edge(b4) (b3)edge(b2) (b3)edge(b6) (b5)edge(b4) (b5)edge(b6);
        \draw[edge] (b1)edge(b7) (b3)edge(b7) (b7)edge(b9) (b6)edge(b9);
        \draw[edge] (b2)edge(b8) (b4)edge(b8) (b8)edge(b10) (b5)edge(b10);
        \draw[edge] (a9)edge(b9) (a10)edge(b10);
    \end{tikzpicture}
    \qquad
    \begin{tikzpicture}[xscale=0.75,yscale=0.5]
        \begin{scope}[xshift=-6cm]
            \node[vertex] (a1) at (-1,-1) {};
            \node[vertex] (a3) at (0,-1) {};
            \node[vertex] (a5) at (1,-1) {};
            \node[vertex] (a2) at (-1,1) {};
            \node[vertex] (a4) at (0,1) {};
            \node[vertex] (a6) at (1,1) {};
            \node[vertex] (a7) at (-0.5,-2) {};
            \node[vertex] (a8) at (-0.5,2) {};
            \node[vertex] (a9) at (1.5,-2) {};
            \node[vertex] (a10) at (1.5,2) {};
            \node[vertex] (a11) at (2.5,-2.5) {};
            \node[vertex] (a12) at (2.5,2.5) {};
        \end{scope}
        \begin{scope}[xscale=-1]
            \node[vertex] (b1) at (-1,-1) {};
            \node[vertex] (b3) at (0,-1) {};
            \node[vertex] (b5) at (1,1.5) {};
            \node[vertex] (b2) at (-1,1) {};
            \node[vertex] (b4) at (0,1) {};
            \node[vertex] (b6) at (1,-1.5) {};
            \node[vertex] (b7) at (2,2) {};
            \node[vertex] (b8) at (2,-2) {};
        \end{scope}
        \draw[edge] (a1)edge(a2) (a1)edge(a4) (a3)edge(a2) (a3)edge(a4) (a5)edge(a6);
        \draw[edge] (a1)edge(a7) (a3)edge(a9) (a5)edge(a7) (a5)edge(a9) (a7)edge(a11) (a9)edge(a11);
        \draw[edge] (a2)edge(a8) (a4)edge(a10) (a6)edge(a8) (a6)edge(a10) (a8)edge(a12) (a10)edge(a12);
        \draw[edge] (b1)edge(b2) (b1)edge(b4) (b1)edge(b6) (b3)edge(b2) (b3)edge(b4) (b3)edge(b6) (b5)edge(b2) (b5)edge(b4);
        \draw[edge] (b5)edge(b7) (b6)edge(b8) (b7)edge(b8);
        \draw[edge] (a11)edge(b8) (a12)edge(b7);
    \end{tikzpicture}
    
\caption{Some examples for cubic Ramanujan graphs with edge-connectivity two and 20 vertices.}
\label{fig:cubicRamanujanTwoconnected}
\end{figure}
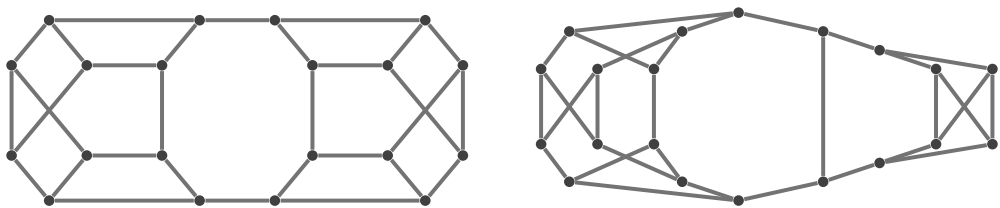

\par\bigskip
\noindent
\textbf{Acknowledgments}\\
Cioab\u{a} is supported by the National Science Foundation grant CIF-1815922.
Dewar and Grasegger were supported by the Austrian Science Fund (FWF): P31888.
Gu is supported by a grant from the Simons Foundation (522728).

\small{
\bibliographystyle{plainurl}
\bibliography{ref}

}

\end{document}